\documentclass{amsart}
\usepackage{graphicx} 

\usepackage{amssymb,amsfonts,amsmath,amsthm,mathtools,bbm,hyperref}
\usepackage{combelow}

\theoremstyle{plain}
\newtheorem{theorem}{Theorem}[section]
\newtheorem*{theorem*}{Theorem}
\newtheorem{corollary}{Corollary}[theorem]
\newtheorem{lemma}[theorem]{Lemma}
\newtheorem*{lemma*}{Lemma}
\newtheorem{proposition}[theorem]{Proposition}

\theoremstyle{definition}
\newtheorem{definition}{Definition}[section]

\theoremstyle{remark} 
\newtheorem{remark}[definition]{Remark}

\usepackage{geometry} 
\usepackage{hyperref}
\usepackage{graphicx}
\usepackage{subcaption}
\usepackage{epstopdf}

\usepackage{enumerate} 
\usepackage{enumitem}

\usepackage{}
\usepackage{xcolor}

\definecolor{darkgreen}{rgb}{0,0.7,0}
\definecolor{darkred}{rgb}{0.5,0,0}
\definecolor{ultramarine}{rgb}{0.07, 0.04, 0.56} 

\usepackage[
    backend=biber,
    style=numeric,
    sorting=nyt,
    sortcites=true, 
    maxbibnames=10,
  ]{biblatex}
\definecolor{darkgreen}{rgb}{0.0, 0.5, 0.0}

\theoremstyle{definition}

\title[A mesoscale regularization: Lennard-Jones drifted SDE vs MCKEAN-VLASOV dynamics ]{Integrability properties and stochastic McKean-Vlasov dynamics with singular Lennard-Jones drift: a mesoscale regularization}

\author{ Ernesto M. Greco}
\address{Ernesto M. Greco:  
Department of Mathematics, University of Milano, Via C. Saldini, 50, 2011, Milano. Italy}
\email{Ernesto.Greco@unimi.it} 
\author{Daniela Morale }
\address{Daniela Morale:  
Department of Mathematics, University of Milano, Via C. Saldini, 50, 2011, Milano. Italy}
\email{Daniela.Morale@unimi.it} 

\begin{document}

\begin{abstract}
 
We study moderately interacting Brownian particle systems driven by a singular force derived from a generalized
Lennard--Jones potential with exponents satisfying \(0<b<a<d-1\).
Within the McKean--Vlasov framework, we establish well-posedness results for the associated nonlinear Fokker--Planck equation and stochastic differential
equation under suitable assumptions on the initial density.
We regularize the interaction at a mesoscale and prove quantitative convergence of the mollified empirical density to the solution of the limiting equation.
Using a semigroup approach, we relate the admissible mesoscale parameter and the convergence rate to the integrability of the force kernel. The condition \(a<d-1\) ensures local integrability of the force and is essential to the framework used here.
In particular, our results do not cover the classical \(12\)--\(6\) Lennard--Jones potential in dimensions \(d=2,3\).

\medskip

 \smallskip
 
\end{abstract}

\maketitle

\section{Introduction}
 
Attractive--repulsive interactions provide a common description of systems in which short-range repulsion competes with attraction at larger distances. The Lennard--Jones potential is a classical example of this structure; see \cite{Wales_2024} for a review. Related interactions
also arise in models of cell adhesion, ion transport, and marble
sulphation \cite{2023_FLR,2025_JMMRU_Arxiv,2025_Mach2023_MRU,2016_MZCJ,2005_Nedea_molecular_dynamics}. In this work, we study Brownian dynamics
driven by a generalized Lennard--Jones force and its approximation
by moderately interacting particle systems.

The generalized Lennard--Jones potential \(\Phi\) and force   \(K\)  are defined   as follows: for any $x \in \mathbb R^d\setminus\{0\}$,  
\begin{equation}\label{eq:JL_potential_force}
\Phi(x)=\epsilon\left(\displaystyle\frac{R_0^a}{a|x|^a} - \frac{R_0^b}{b|x|^b}\right), \qquad K(x)=-\nabla \Phi(x)=  \epsilon\left(\frac{R_0^a}{|x|^{a+1}}-\frac{R_0^b}{|x|^{b+1}} \right) \frac{x}{|x|},
\end{equation}
where $\epsilon,R_0>0$ and $a>b>0$. The parameters $\epsilon$ and $R_0$ determine the strength and characteristic length of the interaction, respectively.
Our analytical results concern dimensions $d\ge2$ and exponents $0<b<a<d-1$. This condition ensures that $K$ is locally integrable. The force has a singularity of order $a+1$ at the origin, while its decay at infinity is governed by $b+1$.
The classical $12$--$6$ Lennard--Jones interaction lies outside this parameter range in dimensions two and three and is not covered by our results.

\smallskip
The first part of the paper provides a systematic analysis of the 
integrability and convolution properties for the generalized
Lennard--Jones force. We identify separately the admissible integrability exponents
$1\le p<d/(a+1)$ near the origin and $q>d/(b+1)$ at infinity, governed by the repulsive and attractive exponents, respectively. We then establish bounds for the operators induced by $K$ and its distributional derivatives, including H\"older and Lipschitz estimates for the convolved force. The analysis distinguishes $a<d-2$, $a=d-2$, and $d-2<a<d-1$, using classical convolution inequalities, singular integral
estimates, and additional density regularity in Bessel potential
spaces. This collection of results makes explicit the dependence of the functional assumptions on the interaction parameters and provides the analytical framework for the subsequent well-posedness and convergence results.

\smallskip
The second part of the paper studies the stochastic dynamics at different scales. Fixing a time horizon  $T>0$, let  $\left(\Omega,\mathcal{F},\mathbb F=(\mathcal{F}_{t})_{t\in [0,T]},\mathbb{P}\right)$ be a filtered probability space and let $W=(W_t)_{t\in[0,T]}$ be an $\mathbb F$-Brownian motion in $\mathbb R^d$. The limiting dynamics is described, at the microscale, by the McKean--Vlasov SDE (MKV-SDE) in $\mathbb R^d$, for $t\in (0,T]$
\begin{equation}\label{SDE:MK Intro}
        dX_t = K\ast u(t,X_t)dt+\sqrt{2}dW_t,  \qquad \mathcal{L}(X_t)=u(t,\cdot)  dx,\\ 
\end{equation}
and $u_0$
is the density of the initial law.  The density  $u$ is  the mild solution to the following  Fokker-Planck PDE, acting at the macroscale, for any
$(t,x)\in (0,T]\times \mathbb R^d$ 
\begin{equation}\label{PDE: FP}
 \partial_t u(t,x) = \Delta u(t,x) -\nabla \cdot \left (u(t,x)(K \ast u (t,x)) \right), \qquad  
u(0,\cdot)=u_0.
\end{equation}  
Our first main result establishes the local well-posedness of these equations. More precisely, let $1<p<d/(a+1)$, let $p'$ be its conjugate exponent, and let $p'\le r<+\infty$. For a probability density $u_0\in L^1\cap L^r(\mathbb R^d)$,
there exists a sufficiently small $T>0$ such that \eqref{PDE: FP} has a unique mild solution in 
\[
C_T^r:=C\bigl([0,T];L^1\cap L^r(\mathbb R^d)\bigr).
\]
Here the intersection space is equipped with the sum of the two
norms. The admissible time horizon depends on the norm of $u_0$,
as quantified in Theorem \ref{Teo: EUMild}; the result is therefore local in time.
The associated McKean--Vlasov SDE admits a strong solution
with marginal density \(u\), and pathwise uniqueness and
uniqueness in law hold within the class of solutions whose
marginal densities belong to \(C_T^r\). Additional assumptions on the integrability and spatial regularity of the initial density yield uniform spatial
Lipschitz estimates for the convolved drift. These assumptions distinguish the regimes below, at, and above \(a=d-2\), and are also used in the particle
approximation estimates.

\smallskip

Given this framework, we rigorously link the aggregative-repulsive dynamics at different scales by considering a mesoscale particle approximation. A mesoscale of order $\alpha\in (0,1)$ is introduced via a regularization of the
interaction by setting
$K_N=K*V_N$ with $V_N(x)=N^{d\alpha}V(N^\alpha x),$
where $V\in C_c^2(\mathbb R^d)$ is a symmetric non-negative probability
density and $\alpha$ is the mesoscale exponent.  

\smallskip
We consider an $N$--particle system in which the particle location processes $X_t^1,\ldots,X^N_t$ are defined on the probability space carrying a family  of independent standard $\mathbb{R}^d$-valued Brownian motions  $\{(W_{t}^{i})_{t\in[0,T]}\}$, for $i=1,\ldots, N$. Particle dynamics are given by the empirical version of equation \eqref{SDE:MK Intro} , i.e. for any 
 $t\in [0,T]$,  
\begin{equation}\label{eq:SDE_Meso}
     d X_t^{i}= K_N\ast\mu^N_t (X_t^i)  dt+ \sqrt{2}dW_t^i,   \quad   1 \le  i \le N,
\end{equation}
with the initial data $(X_0^i)_{i=1}^N$ independent of the Brownian motions.  Equation \eqref{eq:SDE_Meso} depends on the empirical measure
\begin{equation}\label{eq:empirical_measure}
\mu^N_t=\frac{1}{N}\sum_{i=1}^N \varepsilon_{X_t^{i}},
\end{equation}
where $\varepsilon_x$ is the Dirac measure concentrated at $x\in\mathbb R^d$ and $\mu^N_t$ is the empirical measure of  the positions of the $N$ particles at time $t\in[0,T]$. The interaction at the mesoscale, also called  moderate interaction, has been introduced by K. Oelschl\"ager in  \cite{1985_Oelschlaeger,1990_Oelschlaeger} ;
see also \cite{2009_Capasso_Morale}. It  relies on a smoothing of the interaction kernel at the scale  $N^{-\alpha}$, that is, for $|x|$  sufficiently large  compared to $N^{-\alpha}$,  the kernel $K_N$ is very close to $K$ \cite{2009_Capasso_Morale}.  By defining the empirical density as  $u_N(t,\cdot):=V_N \ast \mu_t^N$, for $t\in [0,T]$,  system \eqref{eq:SDE_Meso}, may be re-written, for any $t \in (0,T]$, as  
\begin{equation*}\label{eq:SDE_density}
     dX_t^{i}=\displaystyle  K \ast u_N(t
     ,X_t^{i})dt+ \sqrt{2}dW_t^i, \quad  1 \le i \le N.
\end{equation*}
The admissible range of the mesoscale exponent   $\alpha$ covered by our convergence results depends on the singularity of the interaction force and on the integrability and regularity required of the density. More precisely, let $0<b<a<d-1$, let $r$ be an admissible integrability exponent for the density, and let $r'$ denote its conjugate exponent. For a spatial regularity index $\beta\in[0,1)$, the uniform estimates underlying Proposition \ref{Prop:est_u_N_Bessel}  and Theorem \ref{Teo:StrongConv} require $$ 0<\alpha< \frac{1}{2(\beta+d/r')}. $$
Our second main result gives a quantitative approximation of $u$
by the empirical density of an auxiliary particle system with
a cut-off applied to the drift. Denote this density by
$\widetilde u_N$.
Under the regularity and initial-data assumptions of
Theorem \ref{Teo:StrongConv}, for an admissible finite $r\ge2$,
a spatial regularity index $\beta\in[0,1)$, and $\alpha$ in the above range, 
we prove, for every $m\ge1$ and $0<\varepsilon<\varrho$,
\[
\|\widetilde u_N-u\|_{L^m(\Omega;C_T^r)}
\le C\left(
\|u_N(0,\cdot)-u_0\|_
{L^m(\Omega;L^1\cap L^r(\mathbb R^d))}
+N^{-\varrho+\varepsilon}
\right),
\qquad
\varrho=\min\!\left(\alpha,\frac12-\frac{\alpha d}{r'}\right),
\]
where $r'$ is the conjugate exponent of $r$ and $C$ is independent
of $N$.
Consequently, convergence of the initial densities in the
indicated norm yields convergence of $\widetilde u_N$ in
$L^m(\Omega;C_T^r)$.
A comparison of the systems driven by the same initial positions
and Brownian motions then gives convergence in probability in
$C_T^r$ for $u_N$, the density of the system without the drift
cut-off; see Proposition \ref{Prop:weak_con}.
The interaction in that system remains regularized at the scale
$N^{-\alpha}$.

The restriction on $\alpha$ makes its dependence on the singularity
explicit. Since $r'\le p<d/(a+1)$,
\[
\frac{1}{2(\beta+d/r')}
<
\frac{1}{2(\beta+a+1)}
<
\frac12.
\]
Thus the convergence results cover only part of the initially
introduced range $(0,1)$, subject to the remaining hypotheses.
The bound involving $a$ is an upper ceiling, while the operative
condition involves the chosen exponents $r$ and $\beta$.
Stronger singularities lower this ceiling and require a larger
smoothing length $N^{-\alpha}$ for fixed $N>1$.
The rate reflects the balance between the regularization error
and particle fluctuations. These restrictions arise from the
estimates used in the proof, and no optimality is claimed.
Further results under weaker density regularity assumptions are
given in Corollaries \ref{Cor:strong_con_critic} and \ref{Cor:strong_con_sub}.

\smallskip

The well-posedness and continuum limits of interacting particle systems, including models with regular attractive--repulsive forces, have been extensively studied \cite{2024_anita,2007_morale_Burger_VK,2009_Capasso_Morale,1985_Oelschlaeger}. For systems with Lipschitz drift coefficients, classical approaches establish well-posedness and propagation of chaos under suitable assumptions \cite{1985_Oelschlaeger,Meleard_Coppoletta,Meleard_1996,Flndoli_2019}. We refer to \cite{2014_Jabin} for an overview of mean-field limits for interacting particle systems. Related probabilistic approaches also yield representations of non-conservative PDEs; see \cite{2016_Russo,2026_morale_tarquini_ugolini,2026_morale_tarquini_ugolini_StocDyn,2026_MTU_proc}.

A complementary probabilistic analysis is developed in
\cite{2025_MRU_arxiv}, using a different regularization approach
related to that in \cite{2016_Liu_Yang}. The authors establish
global existence and pathwise uniqueness for a broader class of
diffusion SDEs with generalized Lennard--Jones interactions,
without imposing the local integrability restriction on the
force kernel used here. Their analysis focuses on the
well-posedness of the SDE. The present work studies the
McKean--Vlasov dynamics with convolution drift $K*u$, the
associated nonlinear Fokker--Planck equation, and their
quantitative particle approximation. Our approach therefore
requires explicit integrability and regularity estimates
for the convolved force.

The particle approximation builds on the semigroup framework
developed in \cite{2023_oliveira_richard_tomasevic}.
That work already covers singular attractive--repulsive
interactions, including potentials formed by two power laws,
and develops both Bessel-space estimates and convergence results
without a drift cut-off.
Our analysis specifies the estimates and parameter conditions
for the generalized Lennard--Jones force considered here.
In the regime $d-2<a<d-1$, Theorem \ref{Teo:StrongConv}
gives an error bound directly in $L^1\cap L^r(\mathbb R^d)$, retaining the
fluctuation exponent $1/2-\alpha d/r'$.
The corresponding bound in Theorem 4.2 of \cite{2023_oliveira_richard_tomasevic} is formulated
in $L^1\cap L^\infty(\mathbb R^d)$, with fluctuation exponent
$1/2-\alpha d$.

A further technical refinement concerns the mollifier.
For finite $r\ge2$ and $\beta<1$, our proof of the stochastic
convolution estimate uses only $V\in C_c^2(\mathbb R^d)$.
Interpolation between $L^r(\mathbb R^d)$ and $H^{2,r}(\mathbb R^d)$ controls all spatial
orders needed in the argument, which are at most two.
This identifies the finite regularity sufficient for that
estimate within the established semigroup approach.
The resulting estimates are combined with the mild formulation
of the limiting PDE and a comparison argument that removes the
auxiliary drift cut-off.

\smallskip

This paper is organized as follows. 

\smallskip
 
 Section \ref{Sec:LJ pot} develops the integrability and convolution estimates.  Section \ref{Sec:PDE}  establishes local well-posedness of the limiting PDE and McKean--Vlasov SDE. 
Sections \ref{Sec:ParSys} and  \ref{Sec:SDE-PDE-Uniformly-regularized-drift} introduce the mesoscale particle systems and derive the stochastic convolution and uniform moment estimates. Finally,  Section \ref{sec:convergence}  proves the quantitative approximation and convergence
in probability. The more technical proofs are collected in the appendices.

\section{Convolution operators via the generalized Lennard--Jones force: integrability and continuity properties}\label{Sec:LJ pot}

 The classical Lennard-Jones potential models van der Waals attraction and short-range repulsion. The potential and the corresponding force are given by
\begin{equation}\label{caption:eq:lennard_jones_classical}
\widetilde\Phi(x) = 4\tilde\epsilon \left(\left(\frac{R_0}{|x|}\right)^{12} - \left(\frac{R_0}{|x|}\right)^6\right), \qquad  \widetilde{K}(x) = -\nabla\widetilde\Phi(x)=24\tilde\epsilon R_0^6\left(\frac{2R_0^6}{|x|^{13}}\ -\frac{1}{|x|^7} \right)\frac{x}{|x|}.  
\end{equation}

 Readers interested in applications of this potential may refer to \cite{Wales_2024}.

\begin{figure}[htb]
  \centering
  \begin{subfigure}[t]{0.45\textwidth}
      \centering
      \includegraphics[width=0.8\textwidth]{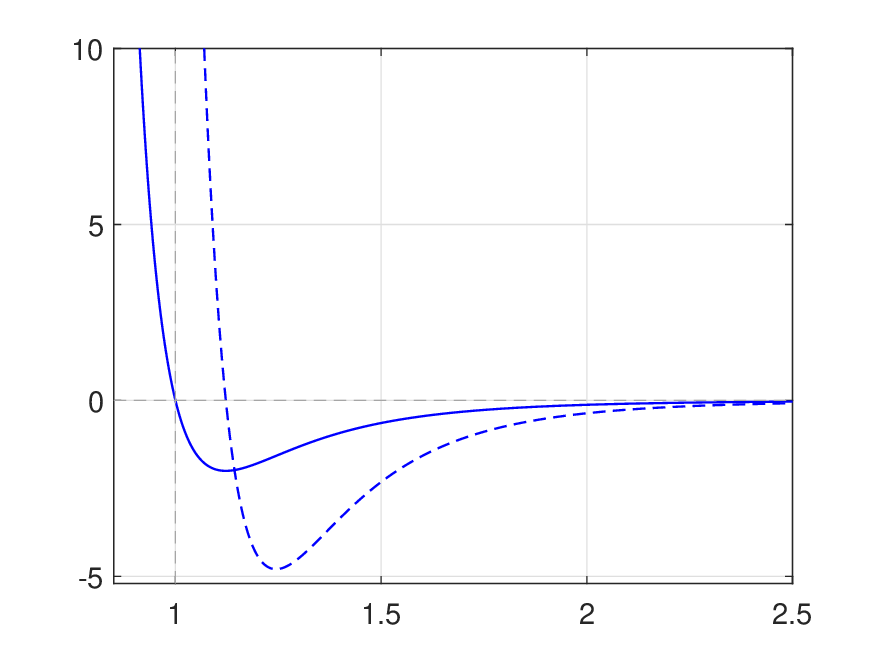}
      \caption{ }
            \label{Pot:LJ}
  \end{subfigure}
  \quad
  \begin{subfigure}[t]{0.45\textwidth}
      \centering
      \includegraphics[width=0.8\textwidth]{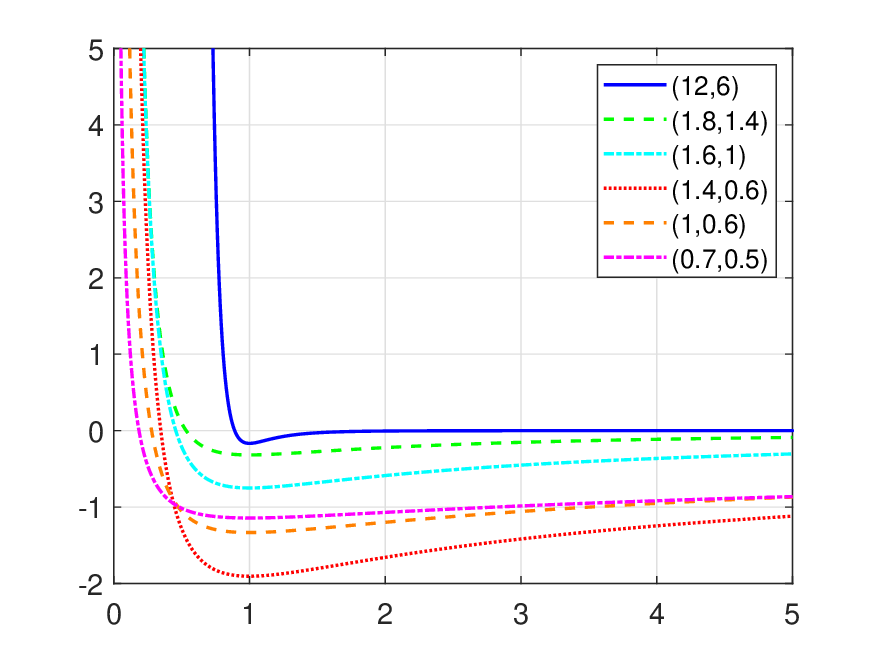}
      \caption { }
      \label{Pot:LJab}
     
  \end{subfigure}
  \caption{(a) Classical Lennard--Jones potential $\widetilde{\Phi}$ (solid line) and the related force $\widetilde{K}$  (dashed line) as in \eqref{caption:eq:lennard_jones_classical}.   (b) Lennard--Jones potential with free parameters  \eqref{eq:JL_potential_force}, for different pairs $(a,b)$.  Parameters: $\epsilon=\tilde\epsilon=2, R_0=1$. The exponents \((a,b)=(12,6)\) are included
for comparison only and are not covered
by our results in dimensions \(d=2,3\).
For each displayed pair, applicability of
the analytical results requires
\(0<b<a<d-1\), together with the additional
assumptions of the relevant theorem. }
\end{figure} 

Figure \ref{Pot:LJ} shows both the behaviour of the classical Lennard--Jones potential and its related force, while Figure \ref{Pot:LJab} illustrates the Lennard--Jones potential with free parameters $(a,b)$, for different pairs of parameters.   
For the classical Lennard--Jones potential, $R_0$ represents the distance at which the potential energy is zero, and the minimum of the potential is equal to $ -\tilde\epsilon$ and is located at $ R_0 2^{{1}/{6}}$. In the case of the Lennard--Jones potential with free parameters, the potential is zero at the location $R_0 \left(\frac{b}{a}\right)^{\frac{1}{a - b}}$, and   its minimum is $  {-\epsilon(a-b)}/{ab}$ and is attained at $R_0$. Hence, in the Lennard--Jones potential with free parameters, from a physical point of view, it would be appropriate to consider $\epsilon=\tilde\epsilon ab/(a-b)$.
A similar argument applies to $R_0$.

  The potential and force in \eqref{eq:JL_potential_force} are referred to as
the generalized Lennard--Jones potential and force.
Their free exponents are restricted to
\(0<b<a<d-1\) in the analytical results below.
The classical potential \eqref{caption:eq:lennard_jones_classical} is presented for comparison;
its exponents \((12,6)\) lie outside this range
in dimensions two and three. 

\subsection{Local integrability properties of the Lennard--Jones force}

 For the generalized Lennard--Jones force  \eqref{eq:JL_potential_force},
local integrability holds precisely when \(a<d-1\).
More generally, the admissible local integrability
exponents depend on the repulsive exponent \(a\),
whereas integrability at infinity is determined
by the attractive exponent \(b\).
The following propositions quantify these properties. To keep this subsection and the following concise, some of the more technical and standard proofs have been deferred to  \ref{sec:appendix_A}.
\begin{proposition}\label{Prop:K_L1}
Let $d \ge 2$. Let $K$ be the generalized Lennard--Jones force \eqref{eq:JL_potential_force} with parameters  $d-1>a>b>0$.   Then,  $K \in L^1_\text{loc}(\mathbb{R}^d)$.  
\end{proposition}
\begin{proof} A detailed proof is in   \ref{appendix:proof:prop_Prop:K_L1}.
\end{proof}

The next result aims to determine sufficient conditions for the local integrability of the kernel $K$. This becomes important since the kernel always appears in a convolution form, so that a natural way to find right estimates is to apply Young's inequality. For this reason, we  find suitable orders of integrability around and outside the singularity, depending on the pair of  parameters $(a,b)$ characterizing  \eqref{eq:JL_potential_force}.

\begin{proposition}[LJ force local integrability]\label{Prop:integrability_pq}
Let $K$  be the   generalized Lennard--Jones force  \eqref{eq:JL_potential_force}.  For a fixed $ \nu \in \mathbb R_+$, let $B(0,\nu)$ denote the ball of radius $\nu$ centred at the origin.  Fixed $d \ge 2$,  and parameters $a$ and $b$ satisfying $d-1>a>b>0$.  
  \begin{enumerate}[label=\roman*)]
      \item If $p\in \displaystyle\left[1,\frac{d}{a+1}\right)$, then
$K\in L^p (B(0,\nu))$. 
\item If   $q\in \left(\displaystyle\frac{d}{b+1},+\infty\right]$, then
 $K\in L^q(B^c(0,\nu))$. 
  \end{enumerate}
\end{proposition}
\begin{proof}
A detailed proof is in \ref{appendix:Prop:integrability_pq}.
\end{proof}
 
 \begin{remark}
      Proposition \ref{Prop:integrability_pq} trivially implies local integrability in Proposition \ref{Prop:K_L1}.
  \end{remark}
\begin{remark}[Scope of the local integrability assumption]
     For the generalized Lennard--Jones force \eqref{eq:JL_potential_force}, local integrability requires  $a<d-1.$
For the classical force \eqref{caption:eq:lennard_jones_classical}, the corresponding condition for local \(L^p\)-integrability, with \(1\le p<+\infty\), is  $13p<d.$ 
In particular, local \(L^1\)-integrability requires \(d\ge14\), so that the classical \(12\)--\(6\) interaction is not covered in dimensions \(d=2,3\).
The restriction \(a<d-1\) is essential to the convolution estimates and regularization used in this paper. It is not asserted to be necessary for every alternative formulation of the dynamics or its continuum limit.
The Bessel potential spaces used in Section \ref{sec:supersingularcase} improve the regularity of the density while retaining this restriction; they do not provide an extension to the classical case.
For the classical potential in physical dimensions, even the repulsive potential \(|x|^{-12}\) is not locally integrable. A treatment of this case would therefore require a different definition or regularization of the singular interaction, rather than a change of solution space alone.
Such a formulation is beyond the scope of the present work.
\end{remark} 

\subsection{Boundedness of  convolution operators}
Let us consider the following notation
\begin{equation}    \label{eq:C_i_nu}
    C_{i,\nu} = \max \left\{
(a+i)R_0^a + (b+i)R_0^b \nu^{\,a-b},\;
(a+i)R_0^a \nu^{-(a-b)} + (b+i)R_0^b
\right\}, \qquad i=1,2.
\end{equation}
 Furthermore, recall that if $f\in {L^s}\cap {L^r}(\mathbb R^d)$, where $1\le s<r\le +\infty$, then
$$\|f\|_{{L^s}\cap {L^r}(\mathbb R^d)}:=\|f\|_{L^s(\mathbb R^d)}+\|f\|_{L^r(\mathbb R^d)}.  $$
\subsubsection{Convolution operator defined by the  generalized Lennard--Jones force \texorpdfstring{$K$}{K} }

Let us consider again the case $d\ge 2$, with $a<d-1$.

\begin{proposition}[LJ force convolution operator]\label{prop:LJ_con_op}  
Let $d \ge 2$ be fixed, and $q \in  \left(\frac{d}{b+1} ,+\infty\right)$. Let $K$ be the   generalized Lennard--Jones force  \eqref{eq:JL_potential_force} with $ d-1>a>b>0$.   
\begin{itemize}
    \item[i)]   If  $r >1$ is such that one of the following two conditions holds
    $$(A) \,\,\,\, \displaystyle\frac{1}{r}=1+\frac{1}{q}  - \frac{1}{p}, \quad \text{with } p \in \left[1,\displaystyle\frac{d}{a+1}\right) , \quad\qquad (B)\,\,\,\, \displaystyle\frac{1}{r}=\frac{1}{q} +1 - \frac{a+1}{d} ,$$
     then,
        $K$ defines a convolution operator, bounded componentwise from $L^1\cap L^{r}(\mathbb R^d)$ to $L^q(\mathbb R^d)$, 
        i.e. there exists $0<C=C(\epsilon,\nu,a,b,d,p,q)$ such that
$$
\|K*f\|_{L^q(\mathbb R^d)}\le  C \|f\|_{{L^1}\cap {L^r}(\mathbb R^d)},
$$
where in case $(A)$, $C= C_{1,\nu}C_{K,p,q}$ with $C_{1,\nu}$ as in \eqref{eq:C_i_nu}  and \begin{equation}\label{eq:def_C_Kpq}
C_{K, p,q}:= 2\max\left(\| K_1\|_{L^p(\mathbb R^d)}, \| {K_2}\|_{L^q(\mathbb R^d)}\right), 
\end{equation}
 where $K_1(x):=\epsilon|x|^{-(a+1)} \, \mathbbm{1}_{B(0, \nu)}(x)$ and $K_2(x):=\epsilon|x|^{-(b+1)} \, \mathbbm{1}_{B^c(0, \nu)}(x)$.
    \item[ii)]Let $p \in \left[1,\frac{d}{a+1}\right)$ and let $p^\prime$ be its conjugate exponent, then   $K$ defines a convolution operator, bounded componentwise from $L^1\cap L^{p^\prime}(\mathbb R^d)$ to $L^\infty(\mathbb R^d)$, i.e.
    \begin{equation}\label{eq:Ckd_K*f}
\|K*f\|_{L^\infty(\mathbb R^d)}\le  C_{1,\nu}C_{K,p,q}   \|f\|_{{L^1}\cap {L^{p^\prime}}(\mathbb R^d)}.
\end{equation}
    \end{itemize}
\end{proposition}
\begin{proof} 
Fix $\nu\in\mathbb R_+$. The decomposition yields   
\begin{equation}\label{eq:stima_K}
    | K(x)| \le \  C_{1,\nu}\left(\frac{\epsilon}{|x|^{a+1}} \, \mathbbm{1}_{B(0, \nu)}(x) + \frac{\epsilon}{|x|^{b+1}} \, \mathbbm{1}_{B^c(0, \nu)}(x)\right)=C_{1,\nu}\left(K_1(x)+K_2(x)\right).
  \end{equation}
   Let  $Tf$ denote the  convolution operator  $Tf(x):= K \ast f(x) = \int_{\mathbb{R}^d} K(x-y) f(y) \, dy,$
    we get
    $$
    |Tf(x)|  \le {\int_{|x-y|\le \nu} \frac{\epsilon\, C_{1,\nu}}{|x-y|^{a+1}} |f(y)| \, dy}
+  {\int_{|x-y|>\nu} \frac{\epsilon\, C_{1,\nu}}{|x-y|^{b+1}} |f(y)| \, dy} =:  C_{1,\nu}\left( I_1(x)+ I_2(x)\right).
$$
For $I_2$, Proposition \ref{Prop:integrability_pq} and  the assumption on $q$, imply that   $K_2 \in   L^q(\mathbb R^d)$.  We may therefore
apply Young’s convolution inequality
$$ 
\|I_2\|_{L^q(\mathbb R^d)}=\|K_2*f\|_{L^q(\mathbb R^d)}
\le \|K_2\|_{L^q(\mathbb R^d)}\,\|f\|_{L^1(\mathbb R^d)}.
$$
Let $r> 1$.  First, suppose that case (\emph{A}) holds.
By Proposition \ref{Prop:integrability_pq}, we again have
$K_1 \in   L^p(\mathbb R^d)$, so that one can again apply   Young's inequality for convolution and get
$$ 
\|I_1\|_{L^q(\mathbb R^d)}=\|K_1 * f\|_{L^q(\mathbb R^d)}
\le \|K_1\|_{L^p(\mathbb R^d)} \|f\|_{L^r(\mathbb R^d)}.
$$
Now,   let us consider the case (\emph{B}). $K_1$ behaves as a Riesz potential of order $d-(a+1)$ and by Hardy--Littlewood--Sobolev  convolution inequality \cite[Theorem 25.2]{Samko_1993}  there exists  $C_{HLS}$ depending only upon $d,r,q,\alpha=d-(a+1)$ such that
$$\|I_1\|_{L^q(\mathbb R^d)}=\|K_1 * f\|_{L^q(\mathbb R^d)}
\le C_{HLS} \|f\|_{L^r(\mathbb R^d)}.
$$
This proves the claim \emph{i)}. 

Property \emph{ii)} follows easily from H\"older's inequality, and by the  interpolation properties of $L^p(\mathbb R^d)$ spaces, since $p^\prime>q^\prime$; indeed, 
\begin{align*}
    |Tf(x)|  & \le C_{1,\nu}\left(\|K_1\|_{L^p(\mathbb R^d)}\|f\|_{L^{p^\prime}(\mathbb R^d)}+ \|K_2\|_{L^q(\mathbb R^d)}\|f\|_{L^{q^\prime}(\mathbb R^d)}\right) \\
    & \le C_{1,\nu}C_{K,p,q}\|f\|_{L^1 \cap L^{p^\prime} (\mathbb R^d)}. 
\end{align*}
\end{proof}

\begin{remark}
Condition (\emph{B}) is the  limit version of Condition (\emph{A}). This happens since
$p = {d}/{(a+1)}$ 
is the critical integrability exponent for the singular behaviour of $K$ at the origin: for $p < {d}/{(a+1)}$ 
the singular kernel near \(0\) belongs to \(L^p(\mathbb R^d)\), while for $p \ge {d}/{(a+1)}$  
it does not; hence,   in the limiting case one needs Hardy--Littlewood--Sobolev (HLS) type arguments or weak spaces.
\end{remark}

\subsubsection{Convolution operators defined by \texorpdfstring{$\nabla K$}{nabla K}. The sub-singular case \texorpdfstring{$d \geq 3$}{d ge 3}, \texorpdfstring{$a\leq d-2$}{a le d-2}}
Whenever we restrict the study to a space dimension $d\ge 3$, and impose a stronger restriction on the repulsive exponent $a$, by considering $a\le d-2$,  we can  deduce further regularity on the gradient of the force. In particular, we focus our attention on the integrability properties of the convolution operator defined by  
\begin{equation}\label{eq:HessianMatrixK}
    \nabla K=K^*\frac{I_d}{|x|}+\left(|x|\frac{dK^*}{d|x|}-K^*\right) \frac{x \otimes x}{|x|^3},
\end{equation}
where $K^*=\displaystyle \epsilon\left(\frac{R_0^a}{|x|^{a+1}}-\frac{R_0^b}{|x|^{b+1}} \right)$, $I_d$ is the $d$-dimensional identity matrix, and $\otimes:\mathbb{R}^d \times \mathbb{R}^d \to M^{d \times d}$ denotes the tensor product, i.e., $(x \otimes x)_{i,j}=x_ix_j$.

 \begin{remark}\label{Remark:nablaK_subcritical}
The relation \eqref{eq:HessianMatrixK} follows from the fact that, under the condition $0<b<a<d-2$, the classical and distributional derivatives of $K$ coincide.
\end{remark}

\begin{proposition} \label{Prop:Grad_K1}
    Let $d\ge 3$ be fixed. Consider $K$ the  generalized Lennard--Jones force \eqref{eq:JL_potential_force} with  $d-2>a>b>0$.
     \begin{enumerate}[label=\roman*)]
      \item If $    \bar{p}\in  \left[1,\frac{d}{a+2}\right)$, then
$\nabla K\in L^{\bar{p}} (B(0,\nu))$; if   $\bar{q}\in \left( \frac{d}{b+2},+\infty\right]$, then
 $\nabla K\in L^{\bar{q}}(B^c(0,\nu))$. 
\item    Let $\bar{q}\in \left( \frac{d}{b+2},+\infty\right)$. If  $\bar{r} > 1$ is such that one of the following two conditions holds
    $$(\bar{A}) \,\,\,\, \displaystyle\frac{1}{\bar{r}}=1+\frac{1}{\bar{q}}  - \frac{1}{\bar{p}}, \quad \text{with } \bar{p} \in \left[1,\displaystyle\frac{d}{a+2}\right) , \quad\qquad (\bar{B})\,\,\,\, \displaystyle\frac{1}{\bar{r}}=\frac{1}{\bar{q}} +1 - \frac{a+2}{d} $$
  then
 $\nabla K$ defines a convolution operator, bounded componentwise from $L^1\cap L^{\bar{r}}(\mathbb R^d)$ to $L^{\bar{q}}(\mathbb R^d)$, i.e. there exists $0<C=C(\epsilon,\nu,a,b,d,\bar{p},\bar{q})$  such that
$$
\|\nabla K*f\|_{L^{\bar{q}}(\mathbb R^d)}\le   C\|f\|_{{L^1}\cap L^{\bar{r}}(\mathbb R^d)} , 
$$
 where in case $(\bar{A})$, $C=C_{2,\nu}C_{\nabla K,\bar{p},\bar{q}}$, with $C_{2,\nu}$ as in \eqref{eq:C_i_nu} and 
\begin{equation*}
C_{\nabla K, \bar{p},\bar{q}}:= 2\max\left(\|J_1\|_{L^{\bar{p}}(\mathbb R^d)}, \|J_2\|_{L^{\bar{q}}(\mathbb R^d)}\right), 
\end{equation*}
where $J_1(x):=\epsilon|x|^{-(a+2)} \, \mathbbm{1}_{B(0, \nu)}(x)$ and $J_2(x):=\epsilon|x|^{-(b+2)} \, \mathbbm{1}_{B^c(0, \nu)}(x)$.
\item    Let $\left(\bar{p},\bar{q}\right) \in \left[1, \frac{d}{a+2}\right) \times \left(\frac{d}{b+2},+\infty\right)$. Then    $\nabla K$ defines a convolution operator, bounded componentwise from $L^1\cap L^{\bar{p}\,^{\prime}}(\mathbb R^d)$ to $L^\infty(\mathbb R^d)$, where  $\bar{p}\,^\prime$ is the conjugate exponent of $\bar{p}$.
  \end{enumerate}     
\end{proposition}

\begin{proof}

 From \eqref{eq:HessianMatrixK}, we obtain the following estimate
\begin{equation}\label{eq:stima_gradK}
    |\nabla K(x)| \le C_{2,\nu} \left(\frac{\epsilon}{|x|^{a+2}} \, \mathbbm{1}_{B(0,\nu)}(x) + \frac{\epsilon}{|x|^{b+2}} \, \mathbbm{1}_{B^c(0,\nu)}(x)\right).
  \end{equation}
Then,  with similar arguments as in Proposition \ref{Prop:integrability_pq}, the local integrability properties \emph{i)} are obtained. 

Properties \emph{ii)} and \emph{iii)} follow by the same arguments as those in Proposition \ref{prop:LJ_con_op}, using the estimate \eqref{eq:stima_gradK} instead of \eqref{eq:stima_K}.
\end{proof}

The preceding results identify the integrability exponents for the convolution operators
induced by $K$ and $\nabla K$. We may also say something more about continuity by identifying the H\"older spaces to which the operator belongs.

\begin{definition}[H\"older space]
    We denote by $C^\zeta(\mathbb{R}^d)$ the H\"older space of order $\zeta\in (0,1]$, consisting of all functions $g$ defined over $\mathbb{R}^d$, bounded a.e.,  such that the norm 
\begin{equation*}
    \|g\|_{0,\zeta}:= \| g\|_{L^\infty (\mathbb R^d)} +|g|_\zeta,
\end{equation*}
is finite, where 
\begin{equation*}
|g|_\zeta := \sup_{x\neq y \in \mathbb{R}^d} \frac{|g(x) - g(y)|}{|x-y|^\zeta},
\end{equation*}
is the H\"older seminorm.
\end{definition}

\begin{proposition}\label{Prop:Grad_K2-bis} 
Let $d \ge 3$. Consider the   generalized Lennard--Jones force $K$ given in \eqref{eq:JL_potential_force}, with parameters $d-2>a > b > 0$. Let   $q\in  \left(\frac{d}{b+1}, +\infty\right)$. 
\begin{itemize}
    \item[i)]  Let $p\in   \left[1, \frac{d}{a+2}\right)$ and $p^\prime=\frac{p}{p-1}$.   If  $r$ is such that  $\frac{1}{r}=1+\frac{1}{q}-\frac{1}{p}$ , then, for any $f \in L^1 \cap L^{r}(\mathbb{R}^d)$,  $ K \ast f, \nabla K \ast f \in L^q(\mathbb{R}^d).$ Furthermore,  for any $f \in L^1 \cap L^{p^\prime}(\mathbb{R}^d)$,  then $
K \ast f, \nabla K \ast f \in L^\infty(\mathbb{R}^d).
$  
  \item[ii)]Let  $\ \bar{r}=\frac{qd}{d+q(d-(a+2))}$. For any $f \in L^1 \cap L^{\bar{r}}(\mathbb{R}^d)$,  $ K \ast f, \nabla K \ast f \in L^q(\mathbb{R}^d).$
  
\end{itemize}  

 \end{proposition}

\begin{proof}
The properties \emph{i)}  follow from  Propositions \ref{prop:LJ_con_op}  and \ref{Prop:Grad_K1}, since hypotheses (\emph{A}) and ($\bar{A}$) are satisfied, respectively. The point \emph{ii)} is a direct consequence of Propositions \ref{prop:LJ_con_op}  and \ref{Prop:Grad_K1}; indeed, hypothesis  ($\bar{B}$) is  satisfied and, by observing that the exponent $r$ in ($B$) is smaller  than $\bar{r}$,  by  interpolation arguments we derive that if  $f\in L^1\cap L^{\bar{r}}(\mathbb R^d) $, then it also belongs to $  L^1\cap L^{r}(\mathbb R^d)$, we may apply also hypothesis ($B$). 
\end{proof}

Let us conclude this section by  considering the limit case $a=d-2$,  that corresponds to a critical case in which not only  the local integrability of  $\nabla K$ is missing, so that we cannot apply Young's inequality argument, but also the Hardy--Littlewood--Sobolev (HLS) inequality cannot be applied. However, integrability properties of the convolution can still be established, although without a gain
in regularity, since it is a Calderón--Zygmund type kernel \cite[Chapter 5]{Grafakos_2014}.

\begin{proposition}\label{Prop:Grad_K2} 

Let  $d\ge 3$ be fixed, and let $K$ be the   generalized Lennard--Jones force  \eqref{eq:JL_potential_force} with  $ a>b>0$ and   $d-2=a$. Let $\bar{q}\in \left( \frac{d}{b+2},+\infty\right)$, then the distributional Jacobian $\nabla K$ defines a convolution operator, bounded componentwise from $L^1\cap L^{\bar{q}}(\mathbb{R}^d)$ to $L^{\bar{q}}(\mathbb{R}^d)$.   
\end{proposition}

\begin{proof}
    A detailed proof is in  \ref{appendix:Prop:Grad_K2}.
\end{proof}

\begin{remark}\label{Rem:Grad_K2} 
    Following the same reasoning as in Proposition \ref{Prop:Grad_K2-bis} \emph{i)}, letting $q \in \left(\frac{d}{b+1},+\infty\right)$ and $p=1$, then for any $f\in L^1\cap L^{q}(\mathbb R^d)$, $K\ast f$, $\nabla K \ast f \in L^q(\mathbb R^d)$.
\end{remark}

 \begin{proposition}\label{Prop:HoderSpaceSubCri} 
Let $d \ge 3$. Consider the   generalized Lennard--Jones force $K$ given in \eqref{eq:JL_potential_force}, with parameters $a > b > 0$. Let   $q\in  \left(\frac{d}{b+1}, +\infty\right)$. 
\begin{itemize}
    \item[i)] Let $d-2>a$. Let $r>1$ such that it satisfies    points \emph{i)} or \emph{ii)} of Proposition   \ref{Prop:Grad_K2-bis}. If moreover $q >d$, for any $f \in L^1 \cap L^{r}(\mathbb{R}^d)$ then $K \ast f \in C^{\zeta}(\mathbb R^d)$, with  $\zeta = 1 - d/q$, i.e. the following estimate holds
\begin{equation}\label{Est:Convolution_function_inHolder}
    \left\|K \ast f \right\|_{0,\zeta} \leq C\,\|f\|_{L^1 \cap L^{r}(\mathbb{R}^d)}.
\end{equation}
Furthermore, let $r\ge p'$  where $p'$ is the conjugate exponent of $p$ as in Proposition \ref{Prop:Grad_K2-bis} \emph{i)}. For any $f\in L^1 \cap L^r(\mathbb R^d)$, then $$\|K \ast f\|_{0,1}\le C\|f\|_{L^1 \cap L^r(\mathbb R^d)}.$$
\item[ii)] Let $d-2=a$. If moreover $q >d$, for any $f \in L^1 \cap L^{q}(\mathbb{R}^d)$ then $K \ast f \in C^{\zeta}(\mathbb R^d)$, with  $\zeta = 1 - d/q$, i.e. the following estimate holds
\begin{equation}\label{Est:Convolution_function_inHolderSingular}
    \left\|K \ast f \right\|_{0,\zeta} \leq C\,\|f\|_{L^1 \cap L^{q}(\mathbb{R}^d)}.
    \end{equation}
\end{itemize}  
\end{proposition}
\begin{proof}
   For the case $d-2>a$, from either point \emph{i)} or \emph{ii)} of Proposition \ref{Prop:Grad_K2-bis}, $K*f\in W^{1,q}(\mathbb R^d)$, with $q>d$. Hence we may apply Morrey's inequality and get the claim. For a more detailed  proof  of part \emph{i)}, we refer to \cite[Lemma 5.1]{2023_oliveira_richard_tomasevic}.
    By requiring $r \ge p'$, the Lipschitz continuity of the convolution operator follows directly from the definition of the H\"older norm in conjunction with Proposition \ref{Prop:Grad_K2-bis}, case \emph{i)}.
   The case $d-2=a$ follows from Remark \ref{Rem:Grad_K2}, together with the same arguments as above. 

\end{proof}

\subsubsection{Boundedness of  convolution operators: the super-singular case \texorpdfstring{$a \in (d-2, d-1)$}{a in (d-2, d-1)}}\label{sec:supersingularcase}

\ To treat the regime \(d-2<a<d-1\), we impose additional regularity on the density through Bessel potential spaces. These spaces allow us to obtain estimates for the spatial derivatives of the convolved force while retaining
\(K\in L^1_{\mathrm{loc}}(\mathbb{R}^d)\). We recall the definition and relevant properties of Bessel potential spaces below. The interested reader may refer to \cite{Triebel_1978} for a more detailed discussion.

\begin{definition}[Schwartz functions]
    A $C^\infty$ complex-valued function $f$ on $\mathbb{R}^d$ is called a \emph{Schwartz function} if for any pair of multi-indices $\alpha$ and $\beta$ there exists a non-negative constant $C_{\alpha,\beta}$ such that the \emph{Schwartz seminorm}  of $f$ is finite
\begin{equation*}
    \sup_{x \in \mathbb R^d}\left\|x^\alpha D^\beta f(x)\right\|=C_{\alpha,\beta}<+\infty.
\end{equation*}  
Hence, a   Schwartz function is a smooth function that decays faster than any polynomial at infinity, together with all its derivatives. 
The set of all Schwartz functions on $\mathbb{R}^d$ is denoted by $\mathcal{S}(\mathbb{R}^d)$.   $\mathcal{S}'(\mathbb{R}^d)$ denotes its dual space, that is the space of \emph{tempered distributions}. 
\end{definition}

Tempered distributions act continuously on rapidly decreasing test functions, require at most polynomial growth at infinity and are stable under Fourier transform \cite{Triebel_1983}.

\begin{definition}[Bessel Potential Space]
Let  $p \in (1,+\infty)$ and $\beta \in \mathbb R$. The \emph{Bessel potential space}   $H^{\beta,p}(\mathbb{R}^d)$ is the following  subspace of tempered distributions  
\begin{center}
$H^{\beta,p}(\mathbb R^{d}):= \left\{ u \in \mathcal{S}'(\mathbb{R}^d); \;  \mathcal{F}^{-1}\left( \left(1+|\cdot|^{2}\right)^{\frac{\beta}{2}}\; \mathcal{F} u(\cdot) \right) \in  L^p(\mathbb{R}^d)\right\}$,
\end{center}
where $\mathcal Fu$ denotes the Fourier transform of $u$. The space is endowed with the norm
 \begin{equation*}
 \left\| u \right\|_{H^{\beta,p}(\mathbb R^d)} = \left\| \mathcal{F}^{-1}\left( \left(1+|\cdot|^{2}\right)^{\frac{\beta}{2}
}\; \mathcal{F} u(\cdot) \right) \right\|_{L^p(\mathbb{R}^d)} .
 \end{equation*}
  In particular,
 \begin{equation*}
  \left\Vert u\right\Vert _{H^{0,p}(\mathbb R^d)}=\left\Vert u\right\Vert_{L^p(\mathbb{R}^d)}  \text{ and for any } \beta \le \gamma, \, \ \left\Vert u\right\Vert _{H^{\beta,p}(\mathbb R^d)} \le \left\Vert u\right\Vert_{H^{\gamma,p}(\mathbb R^d)}.
 \end{equation*}
The space $H^{\beta,p}(\mathbb{R}^d)$ is 
associated to the fractional operator $(\mathrm{I}-\Delta)^\frac{\beta}{2}$ acting on $f\in \mathcal{S}'(\mathbb{R}^d)$  as
\begin{equation} \label{eq:fractional_heat_operator}
(\mathrm{I}-\Delta)^\frac{\beta}{2} f := \mathcal{F}^{-1}\left((1+|\cdot|^2)^{\frac{\beta}{2}} \mathcal{F}f \right) .
\end{equation}
\end{definition}

\begin{proposition}\label{Prop:Grad_K3}
    Let $d \ge 2$ fixed. Consider $K$ the   generalized Lennard--Jones force \eqref{eq:JL_potential_force} with $a>b>0$, $a \in (d-2,d-1)$,  $r\in (1,+\infty)$, $\beta \in \mathbb R$ such that $\beta-\displaystyle\frac{d}{r}\in \left(2-d+a,1\right)$.  Then, $\nabla K$ defines a convolution operator, bounded componentwise from $ L^1\cap H^{\beta,r}(\mathbb{R}^d)$ to $L^\infty(\mathbb R^d)$.  
\end{proposition}

\begin{proof}
    A detailed proof is in  \ref{appendix:Prop:Grad_K3}.
\end{proof}

\begin{proposition}\label{Prop:HoderSpaceSuperCri}
    Let $d \ge 2$. Consider the   generalized Lennard--Jones force $K$ given in \eqref{eq:JL_potential_force} with parameters  $a>b>0$ and $a \in (d-2,d-1)$. Let  $q \in \left(\displaystyle\frac{d}{b+1},+\infty\right]$,   $r\in (1,+\infty)$ and $\beta \in \mathbb R$ such that $\beta-\displaystyle\frac{d}{r}\in (2-d+a,1)$. Then, for any  $f\in L^1\cap H^{\beta, r}(\mathbb{R}^d)$, it follows that $K\ast f \in L^q(\mathbb R^d)$,  and $\nabla K \ast f \in L^{\infty}(\mathbb R^d)$. Moreover, by choosing $q=+\infty$, $K \ast f$ is a Lipschitz continuous function.  
\end{proposition}
\begin{proof}

 The first part of the proposition follows directly from Proposition \ref{prop:LJ_con_op} and from the previous proposition, while for the second part one may follow the same arguments as those presented in Proposition \ref{Prop:HoderSpaceSubCri}. 

\end{proof}

Before concluding this section, we address the case $a = d-2$, but this time, by choosing $f$ in a suitable Bessel space in order to ensure the Lipschitz continuity of the operator $K \ast f$.

\begin{proposition}\label{Prop:Grad_K2_Bessel}
  Let $d \ge 3$ fixed. Consider the   generalized Lennard--Jones force $K$ given in \eqref{eq:JL_potential_force}  with $d-2=a>b>0$,   $r\in (1,+\infty)$, $\beta \in \mathbb R$ such that $\beta-\displaystyle\frac{d}{r}\in \left(0,1\right)$, and $r_1\in ( d/(b+2),+\infty]$.  Then, $\nabla K$ defines a convolution operator, bounded componentwise from $ L^1\cap H^{\beta,r}(\mathbb{R}^d)$ to $L^{r_1}(\mathbb R^d)$. Moreover, let $q\in \left(\displaystyle\frac{d}{b+1},+\infty\right]$.     \ For any   $f \in L^1\cap H^{\beta,r}(\mathbb{R}^d) $,   we have that  $K\ast f\in L^{q}(\mathbb R^d)$. Letting $r_1=q=+\infty$, then
  $$\|K\ast f\|_{0,1}\le C\|f\|_{L^1 \cap H^{\beta,r}(\mathbb R^d)}.$$
  
\end{proposition}

\begin{proof}
By applying the same arguments used in Proposition \ref{Prop:Grad_K3} and Proposition \ref{Prop:HoderSpaceSuperCri}, the claim follows.  
\end{proof}

\section{McKean--Vlasov dynamics at the macro-- and microscale }\label{Sec:PDE}

Now, we present a rigorous result of existence and uniqueness of a mild solution of a McKean--Vlasov SDE  involving a singular drift induced by the generalized Lennard-Jones force. We relate this equation to the microscopic description of the dynamics of a typical Brownian particle interacting at a large scale with a field $u$, which evolves at the macroscale via a diffusion-advection PDE.

\emph{Dynamics at the macroscale.} Let  $T>0$ be fixed. Given  the    generalized Lennard--Jones kernel  $K$  as in \eqref{eq:JL_potential_force}, let us consider the following density dynamics at the macroscale, described by the Fokker-Planck equation
\begin{equation}\label{PDE: FP1}
\begin{cases}
 \partial_t u(t,x)  = \Delta u(t,x) -\nabla \cdot \left (u(t,x)(K \ast u (t,x)) \right), \quad (t,x)\in (0,T]\times \mathbb{R}^d;\\
u(0,x) = u_0(x), \hspace{6.5cm}  x\in \mathbb{R}^d.
\end{cases}
\end{equation}

Since $K$   is singular at the origin, the well-posedness is not guaranteed a priori.  

For each $T>0$ and $z\ge1$, let us consider the   space of continuous functions from $[0,T]$ into $L^1 \cap L^z(\mathbb{R}^d)$
\begin{align*} 
 C^{z}_T:=C\left([0,T]; L^1 \cap L^z(\mathbb{R}^d)\right),
\end{align*}
 with the associated norm $ \left\|f\right\|_{C^z_T}:=\sup_{s \in [0,T]} \left\|f(s,\cdot)\right\|_{L^1 \cap L^z(\mathbb{R}^d)}$.

We briefly review the heat semigroup and the core estimates associated with it.

\begin{definition}[Heat semigroup]
    
 The family of operators $\left(e^{t\Delta }\right)_{t\geq 0 }$ denotes the heat  semigroup, defined for $z\ge 1$ and $f \in L^z(\mathbb{R}^d)$. It is given by 
\begin{equation*}
  e^{t\Delta}f(x)   =\int_{\mathbb{R}^{d}} g_{2t}(x-y) \,  f\left(  y\right) \  dy, 
\end{equation*}
where $g_{\sigma^2}$ denotes the $d$-dimensional Gaussian density function
$
g_{\sigma^2}(x)=  (2\pi \sigma^2)^{-{d}/{2}} e^{-{|x|^2}/{2\sigma^2}}.
$
\end{definition}

\begin{proposition}[\cite{Amar_2022}]\label{HeatSemi_Ineq_Lp}
  Fix $T>0$.  Let $z\ge 1$ and consider the space $L^z(\mathbb{R}^d)$.  Then, for any $t\in (0,T]$, $ \left\Vert \nabla g_{2t}\right\Vert _{L^{1}(\mathbb{R}^d)}=C_\Delta t^{-1/2}$,  and   
      \begin{eqnarray}\label{eq:heat_semigroup_estimate}
  \left \Vert \nabla e^{ t \Delta}\right\Vert _{L^z
\to L^z}&:=&\sup_{\substack{f \in L^z(\mathbb{R}^d) \\f\neq0}}\frac{\|\nabla e^{ t \Delta}f\|_{L^z(\mathbb{R}^d)}}{\|f\|_{L^z(\mathbb{R}^d)}}\le \frac{C_\Delta}{\sqrt{t}}.\end{eqnarray} The constant $C_\Delta$ is given by

\begin{equation*}\label{eq:C_Delta}
 C_\Delta = {\Gamma \left(\tfrac{d+1}{2}\right)}/{\Gamma\left(\tfrac{d}{2}\right)}.\end{equation*}
 
\end{proposition}
Let us also  recall two classical estimates related to the fractional operator \eqref{eq:fractional_heat_operator}.
\begin{proposition}[\cite{1983_Pazy}]\label{Es:HeatSemigroup_Bessel}
   Fix $T>0$.  Let $z\ge1$ and $\beta \in (0,1)$. The following estimates for the heat semigroup hold, for any $t\in (0,T]$,
    \begin{eqnarray}\label{eq:stima_Bessel_semigroup} 
    \left\Vert \left(\mathrm{I}-\Delta \right)^{\frac{\beta}{2}} g_{2t}\right\Vert _{L^1(\mathbb{R}^d)}&\le &\frac{C_\Delta^\prime}{\sqrt{t^\beta}};
\\\label{In:BesselHeat}
  \left \Vert \left(\mathrm{I}-\Delta \right)^{\frac{\beta}{2}} e^{ t \Delta}\right\Vert _{L^{z}
\rightarrow L^{z}} &\le& \frac{C_\Delta^\prime}{\sqrt{t^\beta}}. 
\end{eqnarray} 
The constant $C_\Delta^\prime$ may depend on $d,\beta$ and $T$, but is independent of $t\in (0,T]$ and $z$.
\end{proposition}

The McKean--Vlasov PDE \eqref{PDE: FP1} admits a unique mild solution in the sense of the following definition. 
\begin{definition}[Mild solution of order $r$]\label{Def:mild} 
 Let $T>0$, and let $K$ be as in \eqref{eq:JL_potential_force}.  Let $u_0 \in L^1 \cap L^{r}(\mathbb{R}^d)$,  $r\in (1,+\infty)$.
 A function $u$ on $[0,T] \times \mathbb{R}^d$ is said to be a \emph{mild solution  of order $r$} to \eqref{PDE: FP1} if the following conditions hold
\begin{enumerate}[label=(\roman*)]
\item $u\in C^r_T:=C\left([0,T]; L^1 \cap L^{r}(\mathbb{R}^d)\right)$;
\item $u$ satisfies the integral equation
\begin{equation}\label{eq: mild solution}
u(t,\cdot)=e^{t\Delta} u_0-\int_0^t \nabla \cdot \left ( e^{(t-s)\Delta }  (u(s,\cdot) \left(K \ast u(s,\cdot))\right) \right) \ ds, \quad 0 \le t \le T.
\end{equation}
\end{enumerate}
\end{definition}

 \begin{theorem}[Well-posedness of the MKV-Fokker-Planck equation]\label{Teo: EUMild}
Let $d\ge 2$ be fixed. Consider the PDE \eqref{PDE: FP1}   with $K$ given by  \eqref{eq:JL_potential_force} where   $d-1>a>b>0$ are fixed.  
 Let $p \in \left(1,\frac{d}{a+1}\right)$ and let $p^\prime$ be its conjugate exponent. For any $u_0 \in L^1 \cap L^{r}(\mathbb{R}^d), $ with $ r\in [ p^\prime,  +\infty)$, let us consider $T>0$ such that 
      \begin{equation}\label{Lemma contr: Stima1}
 C_{T,u_0}= 1- 4\,C_{\Delta,K,p,q}\sqrt{T}\left\|u_0\right\|_{L^1 \cap L^{r}(\mathbb{R}^d)}\in (0,1),
      \end{equation}
where $C_{\Delta,K,p,q}= 2 d C_{\Delta}C_{1,\nu} C_{K,p,q}$, with the constants as  in \eqref{eq:C_i_nu}, \eqref{eq:def_C_Kpq} and Proposition \ref{eq:C_Delta}. Then, there exists a unique mild solution of order $r$ in the sense of Definition \ref{Def:mild} on $[0,T]$.  Moreover, 
      \begin{equation}\label{eq:bound_u_PDE}
          \left\|u\right\|_{C\left([0,T]; L^1 \cap L^{r}(\mathbb{R}^d)\right)} \le \displaystyle\frac{1- {C_{T,u_0}}^{1/2}}{2C_{\Delta,K,p,q}\sqrt{T}}.
      \end{equation}
    
\end{theorem}

\begin{proof}
The existence follows from a fixed-point argument via   the Banach–Caccioppoli theorem \cite{Evans_2010}. 
Given $C^{r}_T=C\left([0,T]; L^1 \cap L^{r}(\mathbb{R}^d)\right) $, let us consider  the bilinear form $B : C_T^{r} \times C_T^{r} \to C_T^{r}$, defined by
\begin{align*}
    B(u,v)(t) = \int_0^t \nabla \cdot e^{(t-s)\Delta} \left( u(s,\cdot) \left( K \ast v(s,\cdot) \right) \right) \, ds, \quad t \in [0,T]. 
\end{align*}
One easily proves the continuity of $B$, that is \begin{equation} \label{eq:continuity_bilinear_form}
    \| B(u,v)(t) \|_{L^1 \cap L^{r}(\mathbb{R}^d)} \le C_{\Delta,K,p,q} \sqrt{t} \, \|u\|_{C^{r}_T} \|v\|_{C^{r}_T}.
\end{equation}

Indeed, from \eqref{eq:heat_semigroup_estimate},  H\"older's inequality and Proposition \ref{prop:LJ_con_op}, we have

\begin{align*}
    \|B(u,v)(t)\|_{L^1 \cap L^r(\mathbb R^d)} &\le \int_0^t \|\nabla \cdot e^{(t-s)\Delta}\left(u(s,\cdot)(K \ast v(s,\cdot)\right) \|_{L^1 \cap L^r(\mathbb R^d)} \, ds \\ & \le d  \, C_\Delta  \int_0^t\frac{1}{\sqrt{t-s}} \|\left(u(s,\cdot)(K \ast v(s,\cdot)\right) \|_{L^1 \cap L^r(\mathbb R^d)} \, ds \\ & \le  d  \, C_\Delta   \int_0^t\frac{1}{\sqrt{t-s}} \|u(s,\cdot) \|_{L^1 \cap L^r(\mathbb R^d)} \|K \ast v(s,\cdot)\|_{L^\infty(\mathbb R^d)} \, ds \\ & \le  d  \, C_\Delta C_{1,\nu}  C_{K,p,q} \int_0^t\frac{1}{\sqrt{t-s}} \|u(s,\cdot) \|_{L^1 \cap L^r(\mathbb R^d)} \| v(s,\cdot)\|_{L^1 \cap L^r(\mathbb R^d)} \, ds.
\end{align*}

Then, by considering the full trajectory we get the inequality  \eqref{eq:continuity_bilinear_form}. 

Let us consider for any  \begin{equation}\label{Inter:rho}
 \rho \in \left[ \frac{1 - C_{T,u_0}^{1/2}}{2 C_{\Delta,K,p,q} \sqrt{T}}, \frac{1}{2 C_{\Delta,K,p,q} \sqrt{T}} \right), 
 \end{equation}
  the operator $F: \mathcal{B}(\rho) \to \mathcal{B}(\rho)$,  defined on the closed ball  \begin{equation*}
    \mathcal{B}(\rho) = \{ u \in C^{r}_T : \|u\|_{C^{r}_T} \le \rho \};  
\end{equation*}
  for any $ u \in \mathcal{B}(\rho)$  $F(u)$ is such that   $F(u)(t)=e^{t \Delta} u_0-B(u,u)(t)$, for any $t\in [0,T]$. 
Note that by  \eqref{Inter:rho},  we have $$0 < 1 - 2 C_{\Delta,K,p,q} \sqrt{T} \rho \le C_{T,u_0}^{1/2}.$$ By \eqref{Lemma contr: Stima1} and    squaring the previous  inequality  we obtain$$1 - 4 C_{\Delta,K,p,q} \sqrt{T} \rho + 4 C_{\Delta,K,p,q}^2 T \rho^2 \le 1 - 4 C_{\Delta,K,p,q} \sqrt{T} \|u_0\|_{L^1 \cap L^{r}(\mathbb{R}^d)}, $$ which yields  
\begin{equation}\label{Eq:Rel_u0_rho}
    \|u_0\|_{L^1 \cap L^{r}(\mathbb{R}^d)} + C_{\Delta,K,p,q} \sqrt{T} \rho^2 \le \rho.
\end{equation}
It follows from \eqref{Eq:Rel_u0_rho} that $F$ maps $\mathcal{B}(\rho)$ into itself.  In fact, for $u \in \mathcal{B}(\rho)$, we estimate
\begin{align*}
    \|F(u)(t)\|_{L^1 \cap L^{r}(\mathbb{R}^d)} &\le \|e^{t \Delta} u_0\|_{L^1 \cap L^{r}(\mathbb{R}^d)} + \|B(u,u)(t)\|_{L^1 \cap L^{r}(\mathbb{R}^d)} \\
    &\le \|g_{2t}\|_{L^1(\mathbb{R}^d)} \|u_0\|_{L^1 \cap L^{r}(\mathbb{R}^d)} + C_{\Delta,K,p,q} \sqrt{t} \|u\|_{C^{r}_T}^2 \\
    &\le \|u_0\|_{L^1 \cap L^{r}(\mathbb{R}^d)} + C_{\Delta,K,p,q} \sqrt{T} \, \rho^2\le \rho.
\end{align*}

  Furthermore, the constant of the fix-point argument is $q_\rho= 2 \rho C_{\Delta,K,p,q} \, \sqrt{T}$;  by   \eqref{Inter:rho} $q_\rho<1$.   Indeed, for $u, v \in \mathcal{B}(\rho)$, the following estimate holds uniformly for  $t \in [0,T]$
\begin{align*}
    \|F(u)(t) - F(v)(t)\|_{L^1 \cap L^{r}(\mathbb{R}^d)} &= \|B(u,u)(t) - B(v,v)(t)\|_{L^1 \cap L^{r}(\mathbb{R}^d)} \\
    &= \|B(u - v, v)(t) + B(u, u - v)(t)\|_{L^1 \cap L^{r}(\mathbb{R}^d)} \\
    &\le C_{\Delta,K,p,q} \sqrt{T} \left( \|u\|_{C^{r}_T} + \|v\|_{C^{r}_T} \right) \|u - v\|_{C^{r}_T}\\ & \le  q_\rho\,  \|u - v\|_{C^{r}_T} \\
    &<\|u - v\|_{C^{r}_T},
\end{align*}
so  that $F$ is a contraction. We have proved the existence of $u\in C_T^r$, with $\|u\|_{C_T^r}< 1/\left(2C_{\Delta,K,p,q} \sqrt{T}
\right)$.

Uniqueness follows from a  standard argument.
  Let \(u,v\in C([0,T];L^1(\mathbb R^d)\cap L^r(\mathbb R^d))\) be arbitrary mild solutions
with the same initial datum, and let $w(t,x)=u(t,x)-v(t,x)$ such that $w_0(x)=0$. Then, we have 
$$w(t,\cdot) = -\int_0^t \nabla \cdot e^{(t-s)\Delta } \left( w(s,\cdot)(K \ast u(s,\cdot)) + v(s,\cdot)(K \ast w(s,\cdot)) \right) ds.$$
Therefore,  let $C_{T} := 2^{-1} C_{\Delta,K,p,q} \left( \sup_{t \in [0,T]}\|v(t)\|_{{L^1}\cap {L^{r}}(\mathbb R^d)} + \sup_{t \in [0,T]}\|u(t)\|_{{L^1}\cap {L^{r}}(\mathbb R^d)} \right)$, we have 
\begin{align*}
    \|w(t,\cdot)\|_{L^1 \cap L^r(\mathbb R^d)} &\le d C_\Delta \int_0^t \frac{\| w(s,\cdot)(K \ast u(s,\cdot)) + v(s,\cdot)(K \ast w(s,\cdot)) \|_{L^1 \cap L^r(\mathbb R^d)} }{\sqrt{t-s}}ds \\  
    &\le   \frac{C_{\Delta,K,p,q} }{2}    \int_0^t  \frac{\left(\|u(s,\cdot)\|_{{L^1}\cap {L^{r}}(\mathbb R^d)}+ \|v(s,\cdot)\|_{{L^1}\cap {L^{r}}(\mathbb R^d)} \right) \| w(s,\cdot)\|_{L^1 \cap L^r(\mathbb R^d)} }{\sqrt{t-s}}ds \\  
    &\le C_{T}    \int_0^t \frac{\| w(s,\cdot)\|_{L^1 \cap L^r(\mathbb R^d)} }{\sqrt{t-s}}ds.
\end{align*}
Since $\|w(s,\cdot)\|_{L^1 \cap L^r(\mathbb R^d)}$ is continuous and non-negative on $[0,T]$,
  singular Gr\"onwall lemma
\cite[Lemma~7.1.1]{Henry_1981} yields $\|w(t,\cdot)\|_{L^1 \cap L^r(\mathbb R^d)}=0$.
 We conclude that $w \equiv  0$ for all $t \in [0, T]$ and $x \in \mathbb R^d$, which implies $u \equiv v$.
 Thus the fixed point $u\in\mathcal B(\rho)$ is the unique mild
solution in \(C^r_T\): any other mild solution coincides with
\(u\), without any a priori assumption that it belongs to
$\mathcal B(\rho)$, and therefore belongs to that ball as well.
Since this holds for every admissible radius, choosing
the smallest admissible radius in \eqref{Inter:rho} yields the estimate \eqref{eq:bound_u_PDE}.
\end{proof}

 \begin{remark}
    We emphasize that condition \eqref{Lemma contr: Stima1}  in Theorem \ref{Teo: EUMild} links the size of the initial datum $\|u_0\|$ with the time horizon $T$ via the factor $\sqrt{T}$. This implies that for a generic initial datum, our result guarantees the existence of a local-in-time mild solution.
\end{remark}

 \begin{remark}\label{Rem:mass_preservation}
Let $u$ be the mild solution provided by Theorem \ref{Teo: EUMild}.
The mild formulation preserves the total mass.
Indeed, the heat semigroup preserves integrals, while the
divergence term in the mild formulation has zero integral.
The interchange of the spatial and temporal integrals is
justified by the uniform bound
 $$   \sup_{s\in[0,T]}
    \bigl\|u(s,\cdot)(K*u(s,\cdot))\bigr\|_{L^1(\mathbb{R}^d)}
    <+\infty.$$
If $u_0$ is a probability density, the preservation of
non-negativity is established below in Proposition \ref{prop:existence_MKV_SDE},
by identifying \(u(t,\cdot)\) with the marginal density
of the associated stochastic process. Consequently,
$u(t,\cdot)$ is a probability density for every $t\in[0,T]$.
\end{remark}
 The regularity of the solution $u$ may be also improved.
\begin{proposition}\label{Estimate:u}
Under the same hypotheses of Theorem \ref{Teo: EUMild}, if the initial condition  $u_0\in L^1\cap H^{\beta, r}(\mathbb R^{d})$, $0\le \beta<1$ and $T>0$ satisfies \eqref{Lemma contr: Stima1}, then the unique mild solution of the  PDE \eqref{PDE: FP1} satisfies
\begin{equation*}
\sup_{t\in[0,T]} \left\| u(t,\cdot)\right\|_{H^{\beta,r}(\mathbb R^d)} \le \|u_0\|_{H^{\beta,r}(\mathbb R^d)}+2(1-\beta)^{-1}C_\Delta^\prime C_{1,\nu }C_{K,p,q} \|  u\|^2_{C_T^r}T^{(1-\beta)/2}.
\end{equation*}
Moreover, $u\in C([0,T];L^1 \cap H^{\beta,r}(\mathbb R^d))$. 
\end{proposition}

\begin{proof}
Let $u_0\in L^1 \cap H^{\beta,r}(\mathbb R^d)$, by an embedding argument $u_0\in L^1 \cap L^r(\mathbb R^d)$. By  Theorem \ref{Teo: EUMild}, we have that $u\in C([0,T];L^1 \cap L^r(\mathbb R^d))$. By \eqref{In:BesselHeat}, \eqref{eq:Ckd_K*f}, we obtain 
\begin{align*}
    \|u(t,\cdot)\|_{H^{\beta,r}(\mathbb R^d)} &\le \|e^{t\Delta}u_0\|_{H^{\beta,r}(\mathbb R^d)}+\int_{0}^t \|\nabla \cdot e^{(t-s)\Delta}(u(s,\cdot) K\ast u(s,\cdot))\|_{H^{\beta,r}(\mathbb R^d)} \ ds \\ 
    & = \|e^{t\Delta}u_0\|_{H^{\beta,r}(\mathbb R^d)}+\int_{0}^t \| \left(\mathrm{I}-\Delta \right)^{\frac{\beta}{2}}  \nabla \cdot e^{(t-s)\Delta}(u(s,\cdot) K\ast u(s,\cdot))\|_{L^r(\mathbb R^d)} \ ds \\ 
    & \le  \|u_0\|_{H^{\beta,r}(\mathbb R^d)}+C_\Delta^\prime C_{1,\nu }C_{K,p,q} \|  u\|^2_{C_T^r}\int_{0}^t \frac{1}{(t-s)^{(1+\beta)/2}} \ ds \\  &\le \|u_0\|_{H^{\beta,r}(\mathbb R^d)}+2(1-\beta)^{-1}C_\Delta^\prime C_{1,\nu }C_{K,p,q} \|  u\|^2_{C_T^r}t^{(1-\beta)/2}.
\end{align*}
The property $u\in C([0,T];L^1 \cap H^{\beta,r}(\mathbb R^d))$ follows by direct verification, exploiting the strong continuity of the heat semigroup for  $r \in (1,+\infty)$.
\end{proof}

\emph{Dynamics at the microscale.}
Let us consider now the following stochastic differential equation (SDE) 
 \begin{equation}\label{eq:SDE_MKV_1} 
\begin{cases}
        dX_t = K\ast u(t,X_t) dt+\sqrt{2}dW_t, \quad 0< t\le T;\\
         \mathcal{L}(X_0) =u_0(\cdot)  dx,
        \end{cases}
\end{equation}
where   $W=(W_t)_{t\in[0,T]}$ is a Brownian motion. We prove that it is a McKean--Vlasov stochastic differential equation (MKV-SDE) associated with the Fokker-Planck PDE \eqref{PDE: FP1}, that is   for any $t\in (0,T]$, $u(t,\cdot)$ is the marginal density of $X$, hence  \begin{equation}\label{eq:SDE_MKV_2}
        \mathcal{L}(X_t)=u(t,\cdot)  dx, 
\end{equation}
and $u$ is the mild solution of order $r$ to the Fokker-Planck PDE \eqref{PDE: FP1}. 

\begin{definition}[Solution of  the nonlinear 
martingale problem]\label{DEF: MP}
Let  $K$  be the   generalized Lennard--Jones force  \eqref{eq:JL_potential_force}. Let $T>0$ and let  $u_0 \in L^1 \cap L^{r}(\mathbb{R}^d)$ be a probability density function satisfying condition \eqref{Lemma contr: Stima1}.  Let $\mathbb{Q}$ be a probability measure on the canonical space $C([0,T];\mathbb{R}^d)$ equipped with $\sigma(C[0, T];\mathbb{R}^d)$, and   $\mathbb{Q}(t,\cdot)$ its one-dimensional time marginals, for any $t\in [0,T]$. We say that $\mathbb{Q}$ \emph{is a solution of  the nonlinear 
martingale problem} if the following conditions are satisfied
\begin{enumerate}[label=\roman*)]
\item $\mathbb{Q}_0 = q(0,\cdot)dx=u_0 dx$;
\item For any $t\in (0,T]$, $\mathbb Q_t$ has a density $q(t,\cdot)$ w.r.t. Lebesgue measure on $\mathbb{R}^d$. In addition, it satisfies $q\in C\left([0,T];L^1 \cap L^r(\mathbb R^d)\right)$;
\item For any $f \in C_{c}^2(\mathbb{R}^d)$, the process $(M_t)_{t \in[0, T]}$ 
defined as
\begin{equation}\label{MartingalProblem}
    M_t:=f(V_t)-f(V_0)-\int_0^t \left [ \Delta f(V_s)+  \nabla f(V_s) \cdot (K \ast q(s,V_s))\right] \ ds.
\end{equation}
is a $\mathbb Q$-martingale, where $(V_t)_{t \in[0, T]}$ denotes the canonical process associated to  $\mathbb{Q}$.  
\end{enumerate}
\end{definition}

 \medskip
 
The following proposition establishes strong existence and pathwise uniqueness for the MKV-SDE \eqref{eq:SDE_MKV_1}-\eqref{eq:SDE_MKV_2}, within the class of solutions whose marginal densities belong to  $ C\left([0,T]; L^1 \cap L^{r}(\mathbb{R}^d)\right)$.
\begin{proposition}\label{prop:existence_MKV_SDE}
    Let $T>0$ and $u_0 \in L^1 \cap L^{r}(\mathbb{R}^d)$ be a probability density function, with $r$ as in Theorem \ref{Teo: EUMild}, satisfying the condition \eqref{Lemma contr: Stima1}. 
Then, the MKV-SDE \eqref{eq:SDE_MKV_1}-\eqref{eq:SDE_MKV_2} admits a strong solution with marginal densities $u(t,\cdot)$, where $u\in C_T^r= C\left([0,T];L^1 \cap L^r(\mathbb R^d)\right)$ is the unique mild solution of the PDE \eqref{PDE: FP1}. Pathwise uniqueness and uniqueness in law hold within the class of solutions whose marginal densities belong to $C_T^r$.
\end{proposition}
 
\begin{proof}
We give a proof based on Girsanov's theorem for the construction of a weak solution and Veretennikov's strong well-posedness result to establish strong existence and the uniqueness assertions.

 Fix $T>0$. Let $u \in C^r_T $ be the unique mild solution of the PDE \eqref{PDE: FP1}. 
 Define the deterministic drift $ b(t,x)=K*u(t,x)$, for any $(t,x)\in [0,T]\times \mathbb R^d.$
 From Proposition \ref{prop:LJ_con_op} (case \emph{ii)}), and \eqref{eq:bound_u_PDE} we get the uniform boundedness of  drift $b$, i.e.  \begin{equation}\label{eq:uniform_bound_b}
 \sup_{t\in[0,T]} \|b(t,\cdot)\|_{L^\infty(\mathbb R^d)}< \displaystyle\frac{C_{1,\nu}C_{K,p,q}}{2C_{\Delta,K,p,q}\sqrt{T}}=\frac{1}{4 d C_{\Delta} \sqrt{T}}=C_T<+\infty.\end{equation}
 This is a key element for the application of an argument based on the Girsanov's theorem.
 Let us take canonical Wiener measure with initial law $u_0$ or, equivalently, let us consider the purely diffusive SDE
 $$
 dY_t=\sqrt{2}dW_t, \quad \mathcal{L}_{\mathbb P}(Y_0)=u_0(\cdot) dx.$$
We know that $\mathcal{L}_{\mathbb P}(Y_t)$ admits a density $p(t,\cdot)$ under $\mathbb{P}$ (convolution of $u_0$ with the heat kernel $g_{2t}$).
The exponential local martingale 
$$Z_t=\exp\left( \int_0^t  \frac{1}{\sqrt{2}} b(s,Y_s) dW_s -\frac{1}{4} \int_0^t \left| b(s,Y_s)\right|^2 ds\right),$$ is a true martingale because the uniform bound  \eqref{eq:uniform_bound_b} ensures that it satisfies Novikov’s condition
$$
\mathbb E\left[ \exp\left( \frac{1}{4} \int_0^T \left|b(s,Y_s)\right|^2 \, ds\right)   \right]\le \exp\left( \frac{1}{4} C_T^2 T \right) <+ \infty.
$$ Then, a new  probability measure on the underlying filtered probability space  may be defined  such that for $t\in [0,T]$, $$Z_t=\frac{d\mathbb Q}{d\mathbb P}_{|_{\mathcal{F}_t}},$$
and $\mathbb E_{\mathbb P}\left[ Z_t\right]=1$. Thus,    we obtain   
\begin{equation}\label{SDE:Girsanov}
dY_t =b(t,Y_t) dt+ \sqrt{2}\,d\widetilde{W}_t, \qquad \quad  \widetilde{W}_t= W_t-\frac{1}{\sqrt{2}} \int_0^t b(s,Y_s) ds,
\end{equation}
where, by Girsanov's theorem, $\widetilde{W}=\left(\widetilde{W}_t \right)_{t\in [0,T]}$ is a Brownian motion under $\mathbb Q$. Hence, $(Y,\widetilde{W})$ is a   weak solution of the SDE  \eqref{SDE:Girsanov} under $\mathbb{Q}$.
Furthermore, using simple calculations derived by the change of measure, for any $t\in [0,T]$ and any $A\in \mathcal B(\mathbb{R}^d) $, given $\mathbb Q_{t}=\mathcal{L}_{\mathbb Q}(Y_t)$, we have 
$$
\mathbb Q_{t}(A) =\mathbb E_{\mathbb Q}\left[ 1_A(Y_t) \right] =  \mathbb E_{\mathbb P}\left[ Z_t 1_A(Y_t) \right] = 
 \mathbb E_{\mathbb P}\left[ \mathbb E_{\mathbb P}\left[Z_t | Y_t \right] 1_A(Y_t)\right] = \int_A  \mathbb E_{\mathbb P}\left[Z_t | Y_t=x \right] p(t,x) dx, $$
that is $\mathbb Q_{t}$ has a density $q(t,\cdot)$ such that
$ 
q(t,x)= \mathbb E_{\mathbb P}\left[Z_t | Y_t=x \right] p(t,x).$
For any $f\in C_c^2(\mathbb{R}^d)$, by applying It\^o's formula under $\mathbb Q$ to $f(Y_t)$
\begin{equation*}
    \widetilde{M}_t:=f(Y_t)-f(Y_0)-\int_0^t \left [ \Delta f(Y_s)+  \nabla f(Y_s) \cdot  b(s,Y_s)\right] \ ds=\sqrt{2}\int_0^t \nabla  f(Y_s) d\widetilde{W}_s,
\end{equation*} 
is a $\mathbb Q$-martingale. Taking expectations in It\^o's formula shows that \(q\) solves the linear Fokker--Planck equation with bounded drift \(b=K*u\) and initial density \(u_0\). Since \(\|q(t,\cdot)\|_{L^1(\mathbb R^d)}=1\), the corresponding mild
formulation is well defined in \({L^1(\mathbb R^d)}\) and yields \(q\in C([0,T];L^1(\mathbb R^d))\). The function \(u\) satisfies the same mild equation. The heat semigroup gradient estimate and the singular Gr\"onwall inequality therefore imply \(q=u\). Since \(q(t,\cdot)\) is the density of the probability law
\(\mathcal{L}_{\mathbb{Q}}(Y_t)\), the identification \(q(t,\cdot)=u(t,\cdot)\) implies that $u(t,\cdot)$ is a probability density function. 
This also proves the preservation of non-negativity announced in Remark \ref{Rem:mass_preservation}. Thus \(\mathcal L_{\mathbb Q}(Y_t)=u(t,\cdot)\,dx\) for every \(t\in[0,T]\), and the path law \(\mathcal L_{\mathbb Q}(Y)\)
solves the nonlinear martingale problem of Definition~\ref{DEF: MP}. In particular, \(Y\) solves the MKV-SDE \eqref{eq:SDE_MKV_1}--\eqref{eq:SDE_MKV_2}, \cite{Karatzas}. 
 
We now establish strong existence and pathwise uniqueness. By \eqref{eq:uniform_bound_b}, the deterministic drift \(b(t,x):=(K*u(t,\cdot))(x)\) is bounded and Borel measurable. Veretennikov's theorem \cite[Theorem 1]{1981_Veretennikov} therefore yields strong existence and pathwise uniqueness
for the SDE with prescribed  drift \(b\). Uniqueness in law for this SDE implies that its strong solution has the same law as the weak solution constructed above. In particular, its marginal densities are \(u(t,\cdot)\), so it is a strong solution  of the MKV-SDE \eqref{eq:SDE_MKV_1}--\eqref{eq:SDE_MKV_2}.

For any solution with marginal densities in \(C_T^r\), It\^o's formula and the mild formulation imply that these densities solve the PDE \eqref{PDE: FP1}. Theorem~\ref{Teo: EUMild} identifies them with \(u\), so every such solution has the same deterministic drift \(b\). Given two such solutions on the same filtered probability space, driven by the same Brownian motion and with the same
initial random variable, both solve the SDE with the same
deterministic drift \(b\). Pathwise uniqueness for that SDE
therefore implies that they are indistinguishable.
Uniqueness in law for the prescribed-drift SDE also yields
uniqueness in law for the MKV-SDE within this class.
\end{proof}

At the microscopic level,  Proposition \ref{prop:existence_MKV_SDE} establishes strong existence and pathwise uniqueness for the microscopic dynamics within the stated density class.

\smallskip

We next impose additional assumptions on the interaction exponents and the initial density to obtain a uniform spatial Lipschitz bound for the convolved drift. These estimates are  also used to control the regularization error in Section \ref{sec:convergence}.

\begin{proposition}\label{Prop:MK_strong_ex_un}
 Let $(p,q)$ be the exponents given in Proposition \ref{Prop:integrability_pq} with $p\ne 1$, and let $(p^{\prime},q^{\prime})$ be their conjugate exponents. Assume  $u_0\in L^1 \cap  H^{\beta, r}(\mathbb{R}^d)$, with $r \in [ p^{\prime}, +\infty)$, be a probability density function. Consider $T>0$ satisfying the condition \eqref{Lemma contr: Stima1}. Furthermore, suppose that  one of the following constraints is satisfied:
\begin{itemize}
    \item[$H_1$.] let $d\ge3$ fixed, $d-2>a$, $p \in (1, d/(a+2))$, and $\beta=0$; 
    \item[$H_2$.] let $d\ge3$ fixed, $d-2=a$, $r>a+2$,  and $\beta \in (0,1)$ such that $\beta -d/r\in (0,1)$;
    \item[$H_3$.] let $d\ge2$ fixed, $a\in (d-2,d-1)$,  $r>d/(d-1-a)$,   and $\beta \in (0,1)$ such that $\beta -d/r\in (2-d+a,1)$. 
\end{itemize}
 Then, for the unique mild solution $u$ of the PDE \eqref{PDE: FP1}, the convolved drift, defined for any $(t,x)\in [0,T]\times \mathbb R^d$ as $b(t,x):=(K*u(t,\cdot))(x)$, is globally Lipschitz in space, uniformly over the local
existence interval \([0,T]\). More precisely, there exists a finite constant $L$ such that
\begin{equation}\label{eq:global_lipschitz}
\sup_{t\in[0,T]}\left\|\nabla b(t,\cdot) \right\|_{L^\infty(\mathbb R^d)}<L.
\end{equation}
\end{proposition}
\begin{proof}
 Under \(H_1\), the uniform spatial Lipschitz bound follows from Proposition \ref{Prop:HoderSpaceSubCri} (i), together with the bound on \(u\) in \(C_T^r\) provided by Theorem~\ref{Teo: EUMild}.

Under \(H_2\), Proposition \ref{Estimate:u} gives
\[
\sup_{t\in[0,T]}
\|u(t,\cdot)\|_{L^1\cap H^{\beta,r}(\mathbb R^d)}
<+\infty.
\]
Proposition \ref{Prop:Grad_K2_Bessel} therefore yields
\eqref{eq:global_lipschitz}. Under \(H_3\), the same argument applies, using 
 Proposition \ref{Prop:HoderSpaceSuperCri}, instead of   Proposition \ref{Prop:Grad_K2_Bessel}. In each case, the resulting bound is uniform over the
local existence interval \([0,T]\) and implies the asserted global spatial Lipschitz continuity.
\end{proof}

\begin{remark}\label{Re:MK_strong_ex_un}
The condition $r>a+2$ in $H_2$ arises from the requirement that $\beta-d/r\in(0,1)$, which is necessary to apply Proposition \ref{Prop:Grad_K2_Bessel}. However, this condition is always satisfied since, when $a=d-2$, we have $p^\prime \in \left(a+2,+\infty\right]$. Similarly, in $H_3$ we obtain the condition $ r>d/(d-a-1).$
\end{remark}


\section{Brownian particles moderately interacting via a generalized Lennard--Jones force at a mesoscale}\label{Sec:ParSys}

Let $E$ be a Banach space on $\mathbb R^d$ and $m\in \mathbb N$. We denote by   $L^m(\Omega; C([0,T];E))$  the Banach space of $E$-valued stochastic processes $X = (X_t)_{t \in [0,T]}$ which are continuous in time and have a finite $m$-th moment, equipped with the following  norm
$$\|X\|^m_{L^m(\Omega; C([0,T];E))} :=  \mathbb{E}\left[ \sup_{t \in [0,T]} \|X_t\|_{E}^m \right].$$

\subsection{The mesoscale} Let us introduce a mesoscopic scale determined by a rescaling parameter $\alpha \in (0,1)$. In the following we consider a finite number of Brownian particles  interacting at the mesoscale via the   generalized Lennard--Jones force. 
\begin{definition}[Regularized Lennard--Jones force and mollifier]\label{def:mollified_lennard_jones_force}
We define the \emph{regularized Lennard--Jones force} of order $\alpha\in (0,1)$ and scale $N\in \mathbb N$ to be the function $K_N$ defined as \begin{equation} \label{eq:def_KN}
    K_N(x)= K\ast V_N(x), 
\end{equation} 
where  $K$ is given by  \eqref{eq:JL_potential_force} and  $V_N$ is a \emph{mollifier}, that is \begin{equation}\label{eq.Def_VN}
    V_N(x)=N^{d\alpha} \, V(N^\alpha x), \quad \alpha \in (0,1).
\end{equation} The function  $V$  is a symmetric probability density function such that $V\in C^2_c(\mathbb{R}^d)$, the space of continuous and compactly supported functions with continuous derivatives up to second order. We may refer to $K_N$ also as  \emph{ mollified Lennard--Jones force}.
  \end{definition}
  \begin{remark}\label{remark:not_uniform_estimates_V_N}
       Note that from hypothesis  \eqref{eq.Def_VN}, it follows that for any $p\in [1,+\infty]$
\begin{equation}\label{eq:VN_Lp}
\begin{split}
    \|V_N\|_{L^p(\mathbb R^d)}&= N^{\alpha(d-d/p)}\|V\|_{L^p(\mathbb R^d)};\\
    \|\nabla V_N\|_{L^p(\mathbb{R}^d)}
&\le N^{\alpha\left(d+1-\frac{d}{p}\right)} \, \|\nabla V\|_{L^p(\mathbb{R}^d)}.
\end{split}
\end{equation}
Moreover, let $\beta \in [0,1)$ and $p \in (1,+\infty)$, by complex interpolation (see \eqref{Est: scaling_VN} in \ref{Est:S_N}), we get
\begin{equation*}\label{eq:VN_Hbetar}  
\begin{split}
   \|V_N\|_{H^{\beta,p}(\mathbb R^d)} &\le C N^{\alpha(d+\beta-d/p)}\|V\|_{H^{2,p}(\mathbb R^d)} ;\\
   \|\nabla V_N\|_{H^{\beta,p}(\mathbb R^d)} &\le C N^{\alpha\left(d+\beta+1-\frac{d}{p}\right)} \, \| V\|_{H^{2,p}(\mathbb R^d)}.
\end{split}
\end{equation*}
\end{remark}
The parameter $\alpha$ is referred to as \emph{the mesoscale rescaling parameter}, and it governs the effective range of the interaction: the rescaling $V_N$  interpolates between microscopic and macroscopic interaction regimes. The admissible range of $\alpha \in (0,1)$  will be restricted appropriately in the subsequent analysis.
 
\begin{proposition}\label{Prop:KN_bounded}
Let $d \ge 2$. The regularized Lennard--Jones force $K_N$ in \eqref{eq:def_KN} is such that   $ K_N \in C^2(\mathbb{R}^d)\cap L^\infty(\mathbb{R}^d)$, for any finite $N\in \mathbb N$.  

\end{proposition}
\begin{proof} A detailed proof is in  \ref{appendix:KN_properties}.
\end{proof}

\emph{The interacting particle system.} 
Let $T>0$ be the time horizon, let $N\in \mathbb N$ be the total number of particles, and let $W=\{(W_{t}^{i})_{t\in[0,T]}, \  1 \le i \le N \}$ be a family of independent standard $\mathbb{R}^d$-valued Brownian motions defined on a filtered probability space $\left(\Omega,\mathcal{F},(\mathcal{F}_{t})_{t \in [0,T]},\mathbb{P}\right)$.   Denote by $(X^1,\ldots,X^N) =\left(X^1_t,\ldots,X^N_t \right)_{t\in [0,T]}$ the $Nd$-dimensional stochastic process in $(\mathbb R^{Nd},\mathcal{B}_{\mathbb R^{Nd}}),$ modelling the location of the $N$ particles which move randomly, according to the  following dynamics

\begin{equation}\label{SDE:Meso1}
\begin{cases}
dX_{t}^{i}=   \displaystyle\frac{1}{N}\displaystyle\sum_{k=1 }^{N} K_{N}(X_{t}^{i} -X_{t}^{k})  \ dt + \sqrt{2} \ dW_{t}^{i}, \quad 0<t\le T, \quad 1 \le i \leq N; \\
X_0^{i}=\xi^i,
\end{cases}
\end{equation}
the initial conditions  $\xi^i, i=1,\dots,N$  are independent random variables,  independent of $W$. 
By \eqref{eq:def_KN} and \eqref{eq.Def_VN}, particles are assumed to interact at a mesoscopic scale determined by the rescaling parameter $\alpha \in (0,1)$. System \eqref{SDE:Meso1} admits a strong solution, pathwise unique.

\begin{theorem}[Well-posedness of the particle system] \label{teo_exis:uniq_particle_SDEs}
Let $K$ be as in \eqref{eq:JL_potential_force}. For any $d\ge 2$, let $d-1>a>b>0$. Let $N\ge 2$, and for all $i$, let $\xi^i$ be independent random variables $\mathcal{F}_0$-measurable and independent of $W_i, i=1,\dots,N$. The system  \eqref{SDE:Meso1}
    admits a strong solution $(X^1,\ldots,X^N) =\left(X^1_t,\ldots,X^N_t \right)_{t\in [0,T]}$. Moreover, pathwise uniqueness holds.
\end{theorem}
\begin{proof}
  The claim follows from Proposition \ref{Prop:KN_bounded} by a standard application of Veretennikov’s theorem \cite[Theorem 1]{1981_Veretennikov}.
\end{proof}


\subsection{A stochastic PDE for the empirical density} 
 
Let $\mu_N \in C\left([0,T];\mathcal{P}(\mathbb R^d)\right)$
denote the empirical measure-valued process associated with the
$N$-particle system, where $\mathcal{P}(\mathbb R^d)$ is endowed
with the topology of weak convergence. At each time $t\in[0,T]$,
$\mu_N(t)$ is the empirical measure of the particle positions:
\begin{equation*}
\mu_N(t)=\frac{1}{N}\sum_{k=1}^N \varepsilon_{X_t^{k}},
\qquad t\in[0,T],
\end{equation*}
where $\varepsilon_x$ is the Dirac measure concentrated at
$x\in\mathbb R^d$. The corresponding mollified (at the mesoscale) empirical measure, or equivalently, the \emph{empirical density}, is defined by
\begin{equation}\label{eq:def_empirical_density}
    u_N(t,x) = (V_N \ast \mu_N(t))(x) = \frac{1}{N} \sum_{k=1}^N V_N(x - X_t^k),
\qquad x \in \mathbb{R}^d,
\end{equation}
where $V_N$ is the mesoscale rescaling  \eqref{eq.Def_VN}. The empirical density $u_N$ exhibits good regularity properties, under constraints on the parameters characterizing the force.

\begin{proposition}[Empirical density regularity]\label{prop:uN_regularity} Let $K$ be as in \eqref{eq:JL_potential_force}, and $r$ as in Theorem \ref{Teo: EUMild}. For any $d\ge 2$, let $d-1>a>b>0$. For any fixed  $N\ge 2$, the empirical density  $u_N$ defined by \eqref{eq:def_empirical_density} shares the following properties:
\begin{enumerate}[label=\roman*)]
    \item  $u_N \in C\left( [0,T]; C_c^2(\mathbb R^d)\right)$;
    \item $K*u_N$ is well defined  and belongs to $ C\left( [0,T]; C^2(\mathbb R^d)\right)$;
    \item for any $t\in [0,T]$  $K*u_N(t,\cdot)$ is Lipschitz continuous.
\end{enumerate}
\end{proposition}
\begin{proof}
    Property i) follows directly from the conditions on the rescaled kernel \(V_N\) in \eqref{eq.Def_VN} and \eqref{eq:VN_Lp}. Moreover, \(u_N\) is a finite sum of \(C_c^2\) functions and therefore belongs to \(L^p\) for any \(1\leq p\le  + \infty\). Furthermore, since $K$ is also locally integrable by Proposition \ref{Prop:K_L1}, the convolution $K \ast V_N$ is well defined and belongs to $C^2(\mathbb{R}^d)$. This regularization ensures that the double convolution with the empirical measure
\begin{equation*}
K_N \ast\mu_N(t) (X_t^k)= (K \ast V_{N}) \ast \mu_N(t) (X_t^{k})=K \ast (V_{N}\ast \mu_N(t) )(X_t^{k}).
\end{equation*} The identity follows from the associativity of convolution. This proves property $ii)$.

Since $K \ast u_N(t,\cdot) \in C^2(\mathbb R^d)$, a  sufficient condition  for property iii)  is that
$
L: = \sup_{x \in \mathbb{R}^d} \left|\nabla K*u_N(t,x) \right| < +\infty. $
It is easy to see that
$
L \le \|\nabla K_N\|_{L^\infty(\mathbb{R}^d)}, 
$ so it is sufficient to prove that $\|\nabla K_N\|_{L^\infty(\mathbb{R}^d)}<+\infty.$

Now,
\begin{align*}
\nabla K_N(x)=\nabla (K \ast V_N)(x) = \nabla (V_N \ast K)(x) = \nabla V_N \ast K(x) 
= N^{\alpha(d+1)} \int_{\mathbb{R}^d} K(y)\, \nabla V(N^\alpha(x - y)) \, dy.
\end{align*}

Let $s \in [1,+\infty]$. We define the contraction operator
$
D :  (\lambda, f) \in \ \mathbb R_+ \times L^s(\mathbb{R}^d) \to f(\lambda \cdot) \in L^s(\mathbb{R}^d).
$

Fix $\lambda > 0$. If $s = +\infty$, then $D_\lambda$ is an isometry. For $s \in [1,+\infty)$, we have
\begin{align*}
\|D_\lambda(f)\|_{L^s(\mathbb{R}^d)}^s &= \int_{\mathbb{R}^d} |f(\lambda x)|^s\, dx= \lambda^{-d} \|f\|^s_{L^s(\mathbb{R}^d)}.
\end{align*}

  Let $p$ and $q$ be as in Proposition \ref{Prop:integrability_pq} and let $p^\prime$ and $q^\prime$ be their conjugate exponents; we have   
  \begin{align*}
|\nabla K_N(x)| &= N^{\alpha( d+1)} \left| \int_{\mathbb{R}^d} K(y)\, \nabla V(N^\alpha(x - y)) \, dy \right|\\
& \le N^{\alpha(d+1)} \left( \|K\|_{L^p(B(0,\nu))} N^{-d\alpha/p^{\prime}} + \|K\|_{L^q(B^c(0,\nu))} N^{-d\alpha/q^{\prime}} \right) \|\nabla V\|_{L^1 \cap L^{r}(\mathbb{R}^d)} \\
&\le N^{\alpha(d+1)} C_{1,\nu}C_{K,p,q} \left( N^{-d\alpha/p^{\prime}} + N^{-d\alpha/q^{\prime}} \right) \|\nabla V\|_{L^1 \cap L^{r}(\mathbb{R}^d)}.
\end{align*}

Thus, the claim follows from the fact that $V\in C^2_c(\mathbb R^d)$. 
\end{proof}
Proposition \ref{prop:uN_regularity} states regularity conditions for any fixed finite  $N\in \mathbb N$. 
 The particle system \eqref{SDE:Meso1} may be re-written with a clear dependence on the common field $u_N$ as 
\begin{equation}\label{SDE:Meso2}
    dX_{t}^{i}=  K\ast u_N(t,X_{t}^{i})\ dt + \sqrt{2}\ dW_{t}^{i}, \quad t\in[0, T] ,\quad 1 \le  i \le N.
\end{equation}

\begin{proposition}\label{prop:equation_u_N}
Suppose that the assumptions of the Theorem \ref{teo_exis:uniq_particle_SDEs}  are satisfied.  Let   $(X^1,\ldots,X^N) =\left(X^1_t,\ldots,X^N_t \right)_{t\in [0,T]}$ be the unique solution of \eqref{SDE:Meso1}. Then the empirical density $u_N$ is a solution of the following stochastic PDE  (SPDE), defined  for any $(t,x)\in [0,T]\times \mathbb R^d$  \begin{equation}\label{eq:SPDE_u_N}
\begin{split}
u_N(t,x) &= u^N_0(x) + \int_0^t \Delta u_N(s,x) \ ds- \int_0^t \int_{\mathbb R^d}  \nabla V_{N} (x-y)  \left (  K\ast u_N(s ,y)\right) \mu_N(s)(dy) \ ds  \\
& -\frac{\sqrt{2}}{N}\sum_{i=1}^N\int_0^t\nabla V_N(x-X_s^i) \ dW^{i}_s.
\end{split}
\end{equation}
\end{proposition}

\begin{proof}
    The  SPDE \eqref{eq:SPDE_u_N} is derived by a simple application of   It\^o's formula to $u_N$ and the fact that $(X^1,\ldots,X^N)$ is a solution of \eqref{SDE:Meso1}.
\end{proof}

Since by the convolution properties we have $$ \left[\nabla V_{N} *    \left(  K\ast u_N(s ,\cdot)\right) \mu_N(s) \right](x) =  \nabla\cdot  \left[  V_{N} *  \left (  K\ast u_N(s ,\cdot)\right) \mu_N(s) \right](x), $$
the mild form of the equation \eqref{eq:SPDE_u_N} for the empirical density is given by 
 
\begin{equation} \label{MildFormula u}
u_N(t,x)=e^{t\Delta }u^N_0(x) - \int_0^t \nabla \cdot e^{(t-s)\Delta } \left[  V_{N} *  \left (  K\ast u_N(s ,\cdot)\right) \mu_N(s) \right](x)   \ ds+ S_N(t,x),
\end{equation}
where
\begin{equation}\label{eq:StoConInt}
    S_N(t,x):=-\frac{\sqrt{2}}{N}\sum_{i=1}^N\int_0^te^{(t-s)\Delta}\nabla V_N(x-X_s^i) \ dW^{i}_s. 
    \end{equation}

\emph{The initial conditions.}
In the following, we consider the initial condition given below for the particle system.  For any $r\in (1,+\infty)$, $\beta \in[0,1)$ and $m\ge 1$, we have
\begin{eqnarray}
\displaystyle\sup_{1\le i\leq N,\, N\in \mathbb{N}} \|X_0^i\|_{L^m(\Omega)}<+\infty;\label{Initial condition v.a.} \\
 \sup_{N \in \mathbb{N}}   \left\| u_N(0,\cdot) \right\|_{L^m\left(\Omega; L^1 \cap H^{\beta,r}(\mathbb{R}^d)\right)} <+\infty. \label{Initial condition density}
\end{eqnarray}

\subsubsection{The stochastic convolution integral: a priori estimates}
Let us note that,  for any $x\in \mathbb R^d$, and any fixed $N\in \mathbb N$, the stochastic convolution process $S_N$, given by \eqref{eq:StoConInt}, is well defined since the integrand is a square-integrable predictable process; however it is not a martingale because of the dependence of the integrand upon the final time $t$.  Hence, we cannot simply apply a Burkholder–Davis–Gundy (BDG) type inequality in the UMD  (unconditional martingale differences) Banach spaces to control the integral, uniformly in time.
However, by an extension of the Garsia–Rodemich–Rumsey lemma to complete metric space \cite[Appendix A]{2010_Friz}, and a continuity condition for $S_N$  proven in \cite[Proposition B.1.]{2023_oliveira_richard_tomasevic} a bound in  $L^m(\Omega;H^{\beta,r}(\mathbb R^d))$ for this type of stochastic convolution integral has been proved, even though not uniform in $N$.The interested reader may also see \cite{2021_Richard} for the $\mathbb R^d$ case.  We summarize   the overall result in \cite{2023_oliveira_richard_tomasevic} in the following lemma.

\begin{lemma}[\cite{2023_oliveira_richard_tomasevic}]\label{Lemma:ORT}
 Assume that the initial conditions \eqref{Initial condition v.a.} and \eqref{Initial condition density} are satisfied. Let \(V\in C_c^\infty(\mathbb R^d)\) be a non-negative probability density, and let \(K\) be a locally integrable kernel. Consider the particle system \eqref{SDE:Meso2} and its associated stochastic convolution \(S_N\), defined in \eqref{eq:StoConInt}. Assume that there exists a deterministic constant \(M>0\),
independent of \(N\), such that
$$\sup_{N\in \mathbb N} \operatorname*{ess\,sup}_{(\omega,t,x) \in \Omega\times[0,T]\times\mathbb R^d}  |(K*u_N)(\omega, t,x)| \leq M.$$
    
For every \(m\geq1\), \(\gamma\in\mathbb R\), and \(\delta\in(0,1]\), there exists \(C>0\), independent of \(N,s,t\), such that
\[
\|S_N(t,\cdot)-S_N(s,\cdot)\|_{L^m(\Omega;H^{\gamma,2}(\mathbb R^d))}
\leq C|t-s|^{\delta/2} N^{-\frac12
\left(1-\alpha(d+4\delta+2\gamma)\right)}
\] 
for all \(0\leq s\leq t\leq T\) and \(N\geq1\).
For every \(m\geq1\), \(z\in(2,+\infty)\),
\(\beta\in[0,1)\), and \(\varepsilon>0\), there exists
\(C>0\), independent of \(N,t\), such that
\[
\|S_N(t,\cdot)\|_{L^m(\Omega;H^{\beta,z}(\mathbb R^d))}
 \leq C N^{-\frac12+ \alpha(d+\beta-{d}/{z})+\varepsilon}
\]
for all \(t\in[0,T]\) and \(N\geq1\).\\
For every \(m\geq1\), \(z\in[1,+\infty]\), and
\(\varepsilon>0\), there exists \(C>0\), independent
of \(N\), such that 
\[ \|S_N\|_{L^m(\Omega;C([0,T];L^z(\mathbb R^d)))} \leq
    C N^{-\frac12 \left(1-\alpha(d+k(z))\right) +\varepsilon}.
\]
where $ k(z)=\max(0,d(1-2/z)).$
   \end{lemma}
 In Proposition \ref{Prop:est_stochastic_convolution_bessel}
below, we establish a moment estimate for the time supremum
of the stochastic convolution in \(H^{\beta,z}(\mathbb R^d)\),
for \(2\leq z<+\infty\), assuming only \(V\in C_c^2(\mathbb R^d)\).
This requires less regularity of the mollifier than \(V\in C_c^\infty(\mathbb R^d)\), as assumed in \cite{2023_oliveira_richard_tomasevic}.
Our proof uses the BDG inequality directly in the UMD space \(L^z(\mathbb R^d)\), Minkowski's integral inequality in \(L^{z/2}(\mathbb R^d)\),
and complex interpolation for Bessel potential spaces.
This approach controls the spatial regularity orders needed in the argument, all of which are at most two, without passing through a Sobolev embedding from a Hilbert space whose required regularity increases with the dimension.

\smallskip
The remaining estimates needed for the convergence analysis
are deferred to the section devoted to the cut-off particle
system, where the uniform boundedness of the drift provides
the additional control required for those estimates.

\begin{proposition}
    \label{Prop:est_stochastic_convolution_bessel}
  Assume that the initial  conditions  \eqref{Initial condition v.a.} and \eqref{Initial condition density} are satisfied and let  $K$ denote a generic locally integrable kernel and $V \in C^2_c(\mathbb R^d)$. Let $m\ge1$, $z\in [2,+\infty)$, $\beta \in[0,1)$, $\varepsilon>0$. Then there exists $C>0$ such that for any $t \in [0,T]$ and $N \in \mathbb N$, for
    $$
    \overline{S}_N (t,x)= -\frac{\sqrt{2}}{N}\sum_{i=1}^{N}\int_{0}^{t}\left(  \mathrm{I}-\Delta\right)^{\frac{\beta}{2}} e^{\left(  t-s\right)  \Delta} \nabla  
V_N \left( x-X_{s}^{i} \right)   \,dW_{s}^{i},$$ it follows 
 \begin{eqnarray}
 \label{eq:S_N_estimate}
     \left\| {S}_N\right\|_{L^m(\Omega; C([0,T];H^{\beta,z}(\mathbb R^d)))} =    \left\|\overline{S}_N\right\|_{L^m(\Omega; C([0,T];L^z(\mathbb R^d)))}   &\le& CN^{-1/2+\alpha(d+\beta-d/z)+\varepsilon}.
     \end{eqnarray}
     Hence, if  the mesoscale rescaling parameter $\alpha$ is such that $
0<\alpha<[2( \beta +  d/z^\prime) ]^{-1},$
where $z^\prime$ is  conjugate exponent of $z$, then  choosing
$0<\varepsilon\le \frac12-\alpha(\beta+d/z')$ in \eqref{eq:S_N_estimate}
yields a bound uniform in $N$.
\end{proposition}
\begin{proof}
A detailed proof is in the \ref{Est:S_N}.
\end{proof}
\begin{remark}
 We point out that by choosing $z=r$, the condition $r \ge 2$ is automatically satisfied under constraints $H_2$ and $H_3$ of Proposition \ref{Prop:MK_strong_ex_un}, since both hypotheses naturally imply $r > d \ge 2$.  
\end{remark}

\section{Particle system and PDE with a uniformly regularized drift} \label{Sec:SDE-PDE-Uniformly-regularized-drift}
 Our aim is to establish convergence in probability of the empirical density
$u_N$ to the mild solution $u$; that is, for any $\eta > 0$
$$\lim_{N\to +\infty}\mathbb{P}\left(\|u_N-u\|_{C_T^r}>\eta\right)=0.$$ 
To this end, one needs to estimate moments in the space $C([0,T];L^1\cap L^r(\mathbb R^d))$ uniformly in $N$. However, such an estimate is prevented by the singularity of the drift term in the particle system \eqref{SDE:Meso2}, since  Remark \ref{remark:not_uniform_estimates_V_N} gives drift bounds that depend on $N$. To overcome this difficulty, we apply a cut-off procedure of the drift and an auxiliary particle system, as is often done in the literature \cite{2014_Jabin,2023_oliveira_richard_tomasevic,2025_MRU_arxiv}.   
The cut-off introduced here aims to  address two different requirements.
On the one hand, the cut-off threshold is chosen using the a priori bound on the
limiting force field. Indeed,  inequality \eqref{eq:Ckd_K*f} and Theorem \ref{Teo: EUMild} give
\[
\|K * u\|_{L^\infty([0,T]\times \mathbb{R}^d)}
\;\le\;
C_{1,\nu}\, C_{K,p,q}\, \|u\|_{C_T^r}.
\]
 This bound provides a natural scale for the cut-off threshold.
Given a tolerance  $\eta>0$, we choose this threshold so that the cut-off
remains inactive whenever  the empirical density of the auxiliary system stays
within $\eta$ of $u$ in $C_T^r$. This choice does not itself guarantee closeness to $u$. Convergence estimates  show that this event has
probability tending to one, allowing us to compare the particle
systems with and without the drift cut-off.
Accordingly, we set
\begin{equation}\label{Eq:threshold}
B_{\eta,u}:=C_{1,\nu}C_{K,p,q}\left(\eta+ \|u\|_{C_T^r}\right).
\end{equation}
We now give the precise definition of the cut-off function.
\begin{definition}[$u$-dependent cut-off function]\label{def_f_eta}

We define a $u$-dependent cut-off function $f_\eta:\mathbb R^d\to \mathbb R^d$ as a function that possesses continuous and bounded derivatives up to order two, hence $f_\eta\in C^2_b(\mathbb R^d)$ such that, for any $i\in \{1,\dots, d\}$ 
 $$f_{\eta}(x)_i=f(x_i),$$
 where $f\in C^2_b(\mathbb R)$ satisfies the following conditions
\begin{enumerate}[label=\roman*)]
    \item $f(y)=y, $ for $|y|\le  B_{\eta,u} $;
   \item $f(y)=sgn(y) B_{\eta,u}$,  for $|y|\ge B_{\eta,u}+\bar{\eta}$, $\bar{\eta}>0$;
   \item $\|f^\prime\|_{\infty}\le 1$.
\end{enumerate}

\end{definition}
\begin{remark}
    Under these conditions, it follows that $\|f\|_{\infty} \le B_{\eta,u} + \bar{\eta}$.
\end{remark}
\begin{definition}[Cut-off interaction drift] Given the   generalized Lennard--Jones force $K$, we define the cut-off interaction drift as an $\mathbb R^d$-valued  function $b_\eta$ such that, for  any function $v$ and any $x\in \mathbb R^d$
\begin{equation}\label{eq:def_b_eta}
b_{\eta}(v,x)= f_\eta\left( K*v (x)\right),
\end{equation}
with $f_\eta$ as in Definition \ref{def_f_eta}.
\end{definition}

\begin{remark}[Regular drift PDE and SDE]\label{Re:Regular_drift} The cut-off dynamics coincides with the original one
on the region where both the limiting force field and its particle
approximation are controlled, while guaranteeing  uniform bound of the
cut-off particle system.
  Indeed, for any $i \in \{1,\cdots,d\}$, and $(t,x)\in [0,T] \times \mathbb R^d,$     $$ \left|K \ast u(t,x)\right|_i\le \|K\ast u(t,\cdot)\|_{L^\infty(\mathbb R^d)}\le C_{1,\nu}C_{K,p,q} \|u\|_{C_T^r},$$
  therefore, for any $i$-th drift component out of $d$ 
  \begin{equation*}\label{eq:f_eta_PDE}
     f_\eta(K \ast u(t,x))_i=K \ast u(t,x)_i.
       \end{equation*}
Hence, if we consider the auxiliary cut-off PDE and SDE corresponding to the original PDE \eqref{PDE: FP1} and SDE \eqref{eq:SDE_MKV_1}, respectively, Remark \ref{eq:f_eta_PDE} guarantees that they coincide  with the original ones. By contrast, for a  possible associated particle system the cut-off acts upon the dynamics.
    
\end{remark}

Let  $\left(\widetilde{X}^1,\ldots,\widetilde{X}^N\right) $ be the auxiliary $\mathbb R^{Nd}$-valued stochastic process
solving the SDE system with the cut-off drift \eqref{eq:def_b_eta}, that is,
\begin{equation}\label{SDE:MesoReg}
d\widetilde{X}_t^{i}= b_{\eta}(\widetilde{u}_N(t,\cdot),\widetilde{X}^i_t)  \, dt + \sqrt{2} \,dW_{t}^{i}, \quad t\in[0, T], \quad 1 \le i \le N,
\end{equation}
where $\widetilde{u}_N(t,\cdot)$ is   the corresponding 
empirical density \eqref{eq:def_empirical_density} with  $\widetilde{\mu}_N(t)$ the associated empirical measure. Hence,    the empirical density $\widetilde{u}_N$   satisfies the following mild formulation 
\begin{equation}\label{MildFormula u2}
\widetilde{u}_N(t,x)=e^{t\Delta }u^N_0(x) - \int_0^t \nabla \cdot e^{(t-s)\Delta } \left[  V_{N} \ast  b_{\eta}(\widetilde{u}_N(s,\cdot), \cdot) \widetilde{\mu}_N(s) \right](x)   \ ds + \widetilde{S}_N(t,x),
\end{equation}
with \begin{equation}\label{eq:StoConInt2}
    \widetilde{S}_N(t,x):=- \frac{\sqrt{2}}{N}\sum_{i=1}^N\int_0^te^{(t-s)\Delta}\nabla V_N(x-\widetilde{X}_s^i) \ dW^{i}_s. 
\end{equation}

\begin{remark}
   We note that since the  drift is Lipschitz continuous and uniformly bounded, by standard arguments the particle system \eqref{SDE:MesoReg} admits a strong pathwise unique solution.
 \end{remark}

\begin{proposition}\label{Prop:est_u_N_Bessel}
Let $m \geq 1$, $\beta\in [0,1)$, and $d\ge2$; let $r \in [2,+\infty)$ satisfying the assumptions in Theorem \ref{Teo: EUMild} and one of the constraints $H_i, i=1,2,3$ in Proposition \ref{Prop:MK_strong_ex_un}. Let  $r^\prime$ be its conjugate exponent. If the initial condition \eqref{Initial condition density} holds and the mesoscale rescaling parameter range satisfies 
\begin{equation}\label{eq:alfaBessel}
0<\alpha<\frac{1}{2( \beta +  d/r^\prime)},
\end{equation}  
then, the following bound holds
\begin{equation*}\label{St: Bessel u}
 \sup_{N \in \mathbb{N}}   \left\| \widetilde{u}_N \right\|_{L^m\left(\Omega;C([0,T]; H^{\beta,r}(\mathbb{R}^d)\right))}   <+\infty.
\end{equation*}
\end{proposition}

\begin{proof}
 Note that from \eqref{eq:def_b_eta} applied to $\widetilde u_N=V_N*\widetilde \mu_N$,  it follows that for any $s \in [0,T]$
\begin{equation}\label{eq:V_Nb_etau_norm}\begin{split}
    \left\| \left[ V_N \ast \left( b_\eta(\widetilde{u}_N(s,\cdot),\cdot) \widetilde{\mu}_N(s) \right) \right](\cdot) \right\|_{L^r(\mathbb R^d)}
    &\le \sqrt{d}\left(B_{\eta,u}+\bar{\eta}\right)\left\|\frac{1}{N} \sum_{i=1}^N V_N(\cdot - \widetilde{X}_s^i)\right\|_{L^r(\mathbb R^d)} \\ 
    &=\sqrt{d} \left(B_{\eta,u}+\bar{\eta}\right)\left\|\widetilde{u}_N(s, \cdot)\right\|_{L^r(\mathbb R^d)}.
    \end{split}
\end{equation}
Furthermore, from the mild formulation of the cut-off density \eqref{MildFormula u2}, applying the fractional  operator  \eqref{eq:fractional_heat_operator} to the mild solution \eqref{MildFormula u2}, and using the heat semigroup estimates \eqref{eq:heat_semigroup_estimate}-\eqref{In:BesselHeat} together with  \eqref{eq:V_Nb_etau_norm}, we obtain
\begin{align*}
\left\Vert \left(  \mathrm{I}-\Delta\right)^{\frac{\beta}{2}} \widetilde{u}_N(t,\cdot) \right\Vert _{L^{r}(\mathbb R^d) }  
 & \le  \left\Vert \left(  \mathrm{I}-\Delta\right)  ^{\frac{\beta}{2}}e^{t\Delta}u^N_0\right\Vert _{L^{r}(\mathbb R^d)} +\left\| \left(  \mathrm{I}-\Delta\right)^{\frac{\beta}{2}}\widetilde{S}_N(t,\cdot)\right\|_{ L^{r}(\mathbb R^d)}  \\
&  + \int_{0}^{t}\left\Vert \left( \mathrm{I} -\Delta\right)  ^{\frac{\beta}{2}}\nabla \cdot
e^{\left(  t-s\right)  \Delta}\left[ V_N \ast b_{\eta}(\widetilde{u}_N(s,\cdot), \cdot)\widetilde{\mu}_N(s)\right](\cdot)   \right\Vert _{ L^{r}(\mathbb R^d)} \, ds  \\
  &  \le \|u_0^N\|_{H^{\beta,r}(\mathbb R^d)}  + \left\Vert\widetilde{S}_N(t,\cdot)\right\Vert _{H^{\beta,r}(\mathbb R^d)}\\
&\ +\int_{0}^{t} \frac{C^\prime_\Delta \sqrt{d}}{(t-s)^{\frac{1+\beta}{2}}} \left\| \bigl[V_N*b_\eta(\tilde{u}_N(s,\cdot),\cdot)\widetilde{\mu}_N(s)\bigr](\cdot) \right\|_{L^r(\mathbb{R}^d)} \,ds \\ &\le \|u_0^N\|_{H^{\beta,r}(\mathbb R^d)}  + \left\Vert\widetilde{S}_N(t,\cdot)\right\Vert _{H^{\beta,r}(\mathbb R^d)}   \\ &+  C^\prime_\Delta {d}\left(B_{\eta,u}+\bar{\eta} \right) \ \int_{0}^{t} \frac{1}{(t-s)^{\frac{1+\beta}{2}}} \ \left\|  
\widetilde{u}_N(s,\cdot)  \right\| _{L^r(\mathbb R^d)} \ ds.
\end{align*}
 To show that this yields a bound uniform in $N$, we proceed in two sequential steps. We start by noticing that the initial datum norm is finite uniformly in $N$ by \eqref{Initial condition density}. Furthermore, by Proposition \ref{Prop:est_stochastic_convolution_bessel}  the stochastic convolution integral term is bounded uniformly in $N$ with $\varepsilon\le 1/2-\alpha(\beta+d/r')$ in \eqref{eq:S_N_estimate}. 

First, let $\beta = 0$. By applying the singular Gr\"onwall lemma \cite[Lemma 7.1.1]{Henry_1981} and by taking the supremum over $t \in [0,T]$, we obtain
\begin{equation*}
    \sup_{t \in [0,T]} \|\widetilde{u}_N(t,\cdot)\|_{L^r(\mathbb R^d)} \le C_1 \left( \|u_0^N\|_{L^r(\mathbb R^d)} + \sup_{t \in [0,T]} \|\widetilde{S}_N(t,\cdot)\|_{L^r(\mathbb R^d)} \right),
\end{equation*}
where $C_1 := E_{1/2,1}\!\left(C^\prime_\Delta  d\sqrt{\pi}(B_{\eta,u}+\bar{\eta}) T^{1/2}\right)$ is independent of $N$, and $E_{a,b}(z) = \sum_{n=0}^{+\infty} \frac{z^n}{\Gamma(an+b)}$ denotes the Mittag--Leffler function. Taking the $L^m(\Omega)$ norm on both sides, by \eqref{Initial condition density} and Proposition \ref{Prop:est_stochastic_convolution_bessel} the claim follows.

Second, let $\beta \in (0,1)$. By taking the supremum over $t \in [0,T]$ and applying the $L^m(\Omega)$ norm to both sides, we obtain 
\begin{align}\label{Est:widetilde_uN_bounded}   
\begin{split}
\|\widetilde{u}_N\|_{L^m(\Omega; C([0,T]; H^{\beta,r}(\mathbb R^d)))} &\le \|u_0^N\|_{L^m(\Omega; H^{\beta,r}(\mathbb R^d))} + \|\widetilde{S}_N\|_{L^m(\Omega; C([0,T]; H^{\beta,r}(\mathbb R^d)))} \\
    &  + 2(1-\beta)^{-1} C^\prime_\Delta {d}\left(B_{\eta,u}+\bar{\eta} \right) T^{(1-\beta)/2} \|\widetilde{u}_N\|_{L^m(\Omega; C([0,T]; L^r(\mathbb R^d)))}.
\end{split}
\end{align}
 Finally, the term $\|\widetilde{u}_N\|_{L^m(\Omega; C([0,T]; L^r(\mathbb R^d)))}$ corresponds precisely to the case $\beta = 0$, which we have just proven to be bounded uniformly in $N$. Hence, the claim follows.
\end{proof}

\begin{remark}\label{remark:alpha_singularity}
 The admissible mesoscale exponent is related to the singularity
of the interaction kernel through the integrability requirements. Indeed,  the assumptions in Theorem \ref{Teo: EUMild} and  Proposition \ref{Prop:est_u_N_Bessel},  require the regularity order $r$ of $u$   to satisfy   $r \in [ p^\prime, +\infty)$, with $1<p<{d}/(a+1)$. Consequently, $r^\prime \le p  < {d}/(a+1)$.\ Hence,   ${d}/{r'}>a+1,$ where $a+1$ denotes the order of the singularity of the   generalized Lennard--Jones force.  These relations clarify the role of the restriction
\begin{equation*}
0<\alpha<\frac{1}{2( \beta +  d/r^\prime)}.
\end{equation*} 
 This condition governs the effective interaction range between particles  and  is  directly related to the  singularity order of the interaction force.  Indeed, the force has a singularity of order \(a+1\) at the
origin, and the preceding inequalities imply
$ \alpha < {1}/{(2(\beta+d/r'))}   < {1}/{(2(\beta+a+1))}.$
For fixed \(\beta\), the last upper bound decreases as \(a\)
increases. Thus stronger singularities impose a smaller
upper ceiling on the mesoscale exponent. The actual
admissible range is determined by the chosen exponents
\(r\) and \(\beta\), together with the remaining hypotheses.

\smallskip

In the regime \(d-2<a<d-1\), the additional requirement
$    \beta-{d}/{r}>a+2-d $  implies $ \beta+ {d}/{r'}>a+2,$  
    and therefore $    \alpha<{1}/{(2(a+2))}.$
For fixed \(N>1\), a smaller exponent \(\alpha\) corresponds
to a larger smoothing length \(N^{-\alpha}\). This behaviour is consistent with physical intuition: stronger singularities require a larger mesoscale in order to effectively control particle interactions and to obtain a sufficiently regular weighted average in the drift term.

\end{remark}

\begin{proposition}\label{Prop:widetilde_S_NinL}
Assume that the initial  conditions  \eqref{Initial condition v.a.} and \eqref{Initial condition density} are satisfied and let  $K$ denote a generic locally integrable kernel and $V \in C^2_c(\mathbb R^d)$.   Let $r \in [1,+\infty)$ and $m \ge 1$, $\varepsilon>0$. Then, there exists $C>0$ such that for any $t \in [0,T]$ and $N \in \mathbb N$, 

\begin{equation}\label{Est:SN_Leb}
    \left\|\widetilde{S}_N\right\|_{L^m(\Omega; C([0,T];L^r(\mathbb R^d)))}\le C N^{-1/2+\alpha d/\gamma(r)+\varepsilon},
    \end{equation}
where $\widetilde{S}_N$ is defined in \eqref{eq:StoConInt2} and  $\gamma(r)=2 \mathbbm 1_{[1,2)}(r)+ r'\mathbbm 1_{[2,+\infty)}(r).$
\end{proposition}
\begin{proof}
   A detailed proof is in the \ref{Est:S_N}.
\end{proof}

\section{Meso-Macroscale: convergence for the empirical density}\label{sec:convergence}

In the following, we aim to link the aggregative-repulsive dynamics at different scales: at the macroscale, the typical Brownian particle interacts with a field $u$, whose evolution is described by the Fokker-Planck PDE \eqref{PDE: FP1}, while at the microscale  a finite number $N \in \mathbb{N}$ of particles move randomly according to \eqref{SDE:Meso1} and pairwise interact at the mesoscale of order $\alpha$, with $\alpha $ satisfying \eqref{eq:alfaBessel}.  The link between these different scales is proved by showing the convergence in probability of the empirical particle density \eqref{eq:def_empirical_density} associated with the particle system to the unique mild solution of the Fokker--Planck equation \eqref{PDE: FP1}.

\subsection{The case of the regularized SDE: moment convergence}

First, we prove the convergence of the $m$-th moment of $\widetilde{u}_N-u$  in the space $C([0,T];L^1\cap L^r(\mathbb R^d))$. Moreover, we determine the rate of convergence with respect to the number of particles, which is linked to the mesoscale parameter.  The conditions $H_1,H_2,H_3$ of Proposition \ref{Prop:MK_strong_ex_un} also play a role in the particle convergence analysis. In Section \ref{Sec:PDE}, the convolution estimates of Section \ref{Sec:LJ pot}, combined with the regularity of the PDE solution, yield a uniform-in-time Lipschitz bound for $K*u$. In the proof of Theorem \ref{Teo:StrongConv}  below, the same convolution estimates are applied to the mollified empirical density $\tilde u_N$. Together with the uniform moment bounds of Proposition \ref{Prop:est_u_N_Bessel}, they control the regularization error  \eqref{eq:ForceTerm}, yielding the factor $N^{-\alpha}$ in \eqref{eq:F_bound}.

\begin{theorem}[Quantitative convergence of the cut-off empirical density]\label{Teo:StrongConv} 
Let $d\ge 2$  and  $K$ be given by \eqref{eq:JL_potential_force}, with $d-1> a>b>0$. Let $r$ be as in Theorem \ref{Teo: EUMild} and let $r^\prime$ be its conjugate exponent. Let $u_0 \in L^1 \cap H^{\beta,r}(\mathbb R^d)$, with $\beta\in [0,1)$, be a probability density function, and assume that the initial conditions  \eqref{Initial condition v.a.} and \eqref{Initial condition density} are satisfied. Fix $T>0$ satisfying \eqref{Lemma contr: Stima1}. Suppose that one of the constraints in Proposition \ref{Prop:MK_strong_ex_un} holds,  with   $r\ge 2$ when $H_1$ is satisfied, and that the mesoscale rescaling parameter range \eqref{eq:alfaBessel} is satisfied.

Let the dynamics of the particle system be given by \eqref{SDE:MesoReg}, let $\widetilde{u}_N$ be the empirical density \eqref{MildFormula u2} and $u$ be the mild solution of the PDE \eqref{PDE: FP1}.
Then, for any $\varepsilon>0$ and any $m\ge 1$, and given $\varrho  =  \min\! \left(\alpha ,\frac{1}{2}-  \frac{\alpha d}{r^\prime}  \right)$, there exists a constant $C^\prime>0$ such that for all $N\in\mathbb{N}$
\begin{equation}\label{quasi-continuity_mth_moment}
 \left\| \widetilde{u}_N-u \right\|_{L^m(\Omega; C_T^r)} 
\le C^\prime  \left(\left\|u_N(0,\cdot)- u_0\right\|_{L^m(\Omega; L^1 \cap L^r(\mathbb R^d))} +  N^{-\varrho +\varepsilon}\right).
\end{equation} 
\end{theorem}
\begin{proof}

Let us first note that in the case that one of the constraints $H_2$ or $H_3$ is satisfied,   since $H^{\beta,r}(\mathbb{R}^d)$ is continuously embedded into $L^r(\mathbb{R}^d)$, it follows by interpolation that $L^1 \cap H^{\beta,r}(\mathbb{R}^d) $ is continuously embedded into  $L^1 \cap L^r(\mathbb{R}^d) $.  Thus we may  apply Theorem \ref{Teo: EUMild} also in these cases.

\smallskip

Let $u$ be the unique mild solution of order $r$ to equation \eqref{PDE: FP1} on $[0,T]$. 
Denote
$$ \vartheta(s,x):=\langle \widetilde{\mu}_N(s),  V_N(x-\cdot)   b_{\eta}(\widetilde{u}_N(s,\cdot), \cdot) \rangle,$$
where  
$
\langle \mu,f\rangle  =\int_{\mathbb R^d} f(x)\mu(dx).
$
Using  \eqref{MildFormula u2} for $\widetilde{u}_N$ and \eqref{eq: mild solution} for $u$, we obtain
\begin{align*}
\widetilde{u}_N(t,x) - u(t,x) &= e^{t\Delta} (u^N_{0} - u_{0})(x) - \frac{\sqrt{2}}{N}\sum_{i=1}^N \int_0^t \left[e^{(t-s)\Delta} \nabla V_N (\cdot-\widetilde{X}_s^{i}) \right](x)\ dW^i_s\\
& - \int_{0}^t \nabla \cdot e^{(t-s)\Delta } \left[ \vartheta(s,\cdot) - u(s,\cdot)  (K\ast u(s,\cdot))   \right](x) \ ds,
\end{align*}
 
By adding and subtracting within the heat semigroup 
\begin{equation*}
  \widetilde{u}_N(s,x)\, b_{\eta}(\widetilde{u}_N(s,\cdot), x) =     \langle \widetilde{\mu}_N(s),  V_N (x-\cdot) b_{\eta}(\widetilde{u}_N(s,\cdot), x) \rangle 
\end{equation*}  
 we get
\begin{align*}
\widetilde{u}_N(t,x) - u(t,x) &= e^{t\Delta} (u^N_{0} - u_{0})(x) + \widetilde{S}_N(t,x) +  F(t,x)\\
&+ \int_0^t \nabla \cdot e^{(t-s)\Delta } \left[u(s,\cdot)  (K\ast u(s,\cdot)) - \widetilde{u}_N(s,\cdot) b_{\eta}(\widetilde{u}_N(s,\cdot), \cdot)\right](x) ds,
\end{align*}
where  $\widetilde{S}_N $ denotes the stochastic convolution integral \eqref{eq:StoConInt2} and  for any $(t,x)\in [0,T]\times \mathbb R^d$
\begin{align}\label{eq:ForceTerm}
F(t,x)&:= \int_0^t \nabla \cdot e^{(t-s)\Delta } \left[\int_{\mathbb  R^d}   V_N (\cdot-y) \left(b_{\eta}(\widetilde{u}_N(s,\cdot), \cdot)-b_{\eta}(\widetilde{u}_N(s,\cdot), y)\right)\widetilde{\mu}_N(s)(dy)\right](x)\  ds.
\end{align}
Let  $\eta \in \mathbb R_+$ be fixed and  $z\in\{1,r\}$. 
By  \eqref{eq:heat_semigroup_estimate}, one has
\begin{equation}\label{eq:uN-u_proof_1}
\begin{split}
\| \widetilde{u}_N(t,\cdot) &-u(t,\cdot) \|_{L^{z}(\mathbb R^d)} \le \|e^{t\Delta }(u^N_0- u_0 )\|_{L^{z}(\mathbb R^{d})}  + \left\| F(t,\cdot)\right\|_{L^{z}(\mathbb R^d)}  + \left\|  \widetilde{S}_N(t,\cdot)\right\|_{L^{z}(\mathbb R^{d})}\\
&  + C_\Delta \int_0^t \frac{1}{\sqrt{t-s}} \|(u(s,\cdot)  \left(K\ast u(s,\cdot)\right) - \widetilde{u}_N(s,\cdot) b_{\eta}(\widetilde{u}_N(s,\cdot), \cdot) ) \|_{L^{z}(\mathbb R^d)}  ds\\
&\le \|u^N_0- u_0 \|_{L^{z}(\mathbb R^{d})}  + \left\| F(t,\cdot)\right\|_{L^{z}(\mathbb R^d)}  + \left\|  \widetilde{S}_N(t,\cdot)\right\|_{L^{z}(\mathbb R^{d})}\\
&  + A(t)  .
\end{split}
\end{equation}
First of all, let us estimate the second norm on the right hand side of \eqref{eq:uN-u_proof_1}. Using \eqref{eq:heat_semigroup_estimate}  and the 
positivity of $V_N$, we get
\begin{align*}
\| F(t,\cdot)\|_{L^{z}(\mathbb R^d)} 
&\le dC_\Delta 
\int_0^{t}\frac{1}{\sqrt{t-s}} \left(\int_{\mathbb R^d} \langle\widetilde{\mu}_N(s),V_N (x-\cdot) \left| b_{\eta}(\widetilde{u}_N(s,\cdot ), x)- b_{\eta}(\widetilde{u}_N(s,\cdot), \cdot) \right|\rangle^{z} dx\right)^{\frac{1}{z}} \ ds.
\end{align*}
 We use the Lipschitz continuity of the cut-off function and of $K\ast \widetilde{u}_N$, supplied by Proposition \ref{Prop:HoderSpaceSubCri}, case $i)$, Proposition \ref{Prop:Grad_K2_Bessel} or Proposition \ref{Prop:HoderSpaceSuperCri} according to the regime under consideration. Since $V$ is compactly supported, without loss of generality, we can assume that the support of $V$ is included in the ball centred at the origin with radius $\nu$, we have that $V_N (x-y)\left|y-x\right| \le \nu N^{-\alpha} V_N (x-y)$, we get 
\begin{align*}
\| F(t,\cdot)\|_{L^{z}(\mathbb R^d) } &\le dC_{\Delta,K,p,z} 
\int_0^{t} \frac{\| \widetilde{u}_N(s,\cdot)\|_{L^1 \cap H^{\beta,r}(\mathbb R^d)}}{\sqrt{t-s}} \left(\int_{\mathbb R^d} \langle \widetilde{\mu}_N(s),V_N (x-\cdot) \left|\cdot-x\right|\rangle^{z} \  dx\right)^{\frac{1}{z}} \ ds\\ 
&\leq\frac{dC_{\Delta,K,p,z} \, \nu}{N^{\alpha}} 
\int_0^{t}\frac{\| \widetilde{u}_N(s,\cdot)\|_{L^1 \cap H^{\beta,r}(\mathbb R^d)} }{\sqrt{t-s}} \| \widetilde{u}_N(s,\cdot)\|_{L^{z}(\mathbb R^d)} \ ds.
\end{align*}

Again, since $H^{\beta,r}(\mathbb{R}^d)$ is continuously embedded into $L^r(\mathbb{R}^d)$, it follows by interpolation that $L^1 \cap H^{\beta,r}(\mathbb{R}^d) $ is continuously embedded into  $L^1 \cap L^r(\mathbb{R}^d) $.  Then,
\begin{equation*}
\begin{aligned}
\| F(t,\cdot)\|_{L^{z}(\mathbb R^d)}
&\le \frac{dC_{\Delta,K,p,z} \, \nu}{N^{\alpha}} 
\int_0^{t} \frac{1}{\sqrt{t-s}} \| \widetilde{u}_N(s,\cdot)\|_{L^1 \cap H^{\beta,r}(\mathbb R^d)}^2 \ ds\\
&\le \frac{2\sqrt{T}dC_{\Delta,K,p,z} \, \nu}{N^{\alpha}} \sup_{s \in [0,T]}\| \widetilde{u}_N(s,\cdot)\|^2_{L^1 \cap H^{\beta,r}(\mathbb R^d)}.
\end{aligned}
\end{equation*}

 By also applying Proposition \ref{Prop:est_u_N_Bessel},
\begin{equation}\label{eq:F_bound} 
\begin{aligned}
\left\|F  \right\|_{L^m(\Omega; C([0,T];L^{z}(\mathbb R^d)))}&\le \frac{2\sqrt{T}dC_{\Delta,K,p,z} \, \nu \, }{N^{\alpha}} \left( \mathbb{E}\left[ \sup_{s \in [0,T]}\|\widetilde{u}_N(s,\cdot)\|_{L^1 \cap H^{\beta,r}(\mathbb R^d)}^{2m} \right] \right)^{\frac{1}{m}}\\
&\le \frac{2\sqrt{T}dC_{\Delta,K,p,z} \, \nu \, }{N^{\alpha}} \|\widetilde{u}_N\|^2_{L^{2m}(\Omega;C([0,T];L^1\cap H^{\beta,r}(\mathbb R^d)))} 
\end{aligned}
\end{equation} 

Now we proceed with an estimate of the last term at right hand side of \eqref{eq:uN-u_proof_1}. By Remark \ref{Re:Regular_drift}, and by applying the triangle inequality, we obtain
\begin{equation}\label{Convergence:difference}
\begin{split}
A(t) &\le C_\Delta \int_0^t \frac{1}{\sqrt{t-s}} \left\| b_{\eta}(\widetilde{u}_N(s,\cdot), \cdot)(u(s,\cdot)- \widetilde{u}_N(s,\cdot) ) \right\|_{L^z(\mathbb R^d)} ds\\
&\,\, +C_\Delta \int_0^t \frac{1}{\sqrt{t-s}} \left\|u(s,\cdot) (b_{\eta}(u(s,\cdot), \cdot) - b_{\eta}(\widetilde{u}_N(s,\cdot), \cdot)) \right\|_{L^z(\mathbb R^d)}  ds\\
&\le C_{\Delta} \sqrt{d} \left( B_{\eta,u} + \bar{\eta}\right)\int_0^t \frac{1}{\sqrt{t-s}} \| \widetilde{u}_N(s,\cdot)-u(s,\cdot) \|_{L^z(\mathbb R^d)} \ ds\\
& \,\, + \widetilde{C}_{\Delta,K,p,q} \sup_{s \in [0,T]}\|u(s, \cdot)\|_{ L^{z}(\mathbb R^d)} \int_{0}^t \frac{1}{\sqrt{t-s}} \|u(s,\cdot) - \widetilde{u}_N(s,\cdot) \|_{L^1 \cap L^{r}(\mathbb R^d)} \ ds,
\end{split}
\end{equation}
where $\widetilde{C}_{\Delta,K,p,q}= C_\Delta C_{1,\nu}C_{K,p,q}$.  For the first term, H\"older's inequality uses the bound   $$\|b_{\eta}(\widetilde{u}_N(s,\cdot), \cdot)\|_{L^\infty(\mathbb R^d)}\le \sqrt{d}\left( B_{\eta,u} + \bar{\eta}\right).$$ For the second, we use the Lipschitz continuity of the cut-off function and Proposition \ref{prop:LJ_con_op}.

Therefore, denoting $\widetilde{C}=2C_\Delta\sqrt{d} \left( B_{\eta,u} + \bar{\eta}\right)+ \widetilde{C}_{\Delta,K,p,q} \|u\|_{C_T^r}$ from \eqref{eq:uN-u_proof_1}  for both $z=1$ and $z=r$,  and \eqref{Convergence:difference} it follows that 
\begin{equation*}
\begin{split}
\| \widetilde{u}_N(t,\cdot)-u(t,\cdot) \|_{L^1 \cap L^{r}(\mathbb R^d)} &\le  \|u_N(0,\cdot)- u_0 \|_{L^1 \cap L^{r}(\mathbb{R}^d)} \\
& \, + \| F(t,\cdot)\|_{L^1 \cap L^{r}(\mathbb R^d)}  + \|  \widetilde{S}_N(t,\cdot)\|_{L^1 \cap L^{r}(\mathbb{R}^d)}\\
& \, +\widetilde{C} \int_0^t \frac{1}{\sqrt{t-s}}  \| \widetilde{u}_N(s,\cdot)-u(s,\cdot) \|_{L^1 \cap L^{r}(\mathbb R^d)} \ ds
\end{split}
\end{equation*}

By the Gr\"onwall lemma for singular integrals \cite[Lemma 7.1.1]{Henry_1981}, we obtain 
\begin{equation}
\| \widetilde{u}_N-u \|_{C^{r}_T} \le C^* \left(\|u_N(0,\cdot)- u_0 \|_{L^1 \cap L^{r}(\mathbb{R}^d)} + \| F \|_{C^{r}_T}  + \|  \widetilde{S}_N\|_{C^{r}_T} \right), 
\end{equation}
where $C^* := E_{1/2,1}\!\left(\widetilde{C}\sqrt{\pi}\, T^{1/2}\right)$ is a finite constant independent of $N$, and $E_{1/2,1}$ denotes the Mittag--Leffler function.
 
By taking the $L^m(\Omega)$ norm and by using the bound \eqref{eq:F_bound} together with \eqref{Est:SN_Leb} in Proposition \ref{Prop:widetilde_S_NinL}, we conclude that, for any $\varepsilon>0$, there exists $C^\prime>0$ independent of $N$ such that for any $N\in\mathbb{N}$
\begin{align*}
   \| \widetilde{u}_N-u \|_{L^m(\Omega; C_T^r)}  & \le C^\prime \left(\|u_N(0,\cdot)- u_0\|_{L^m(\Omega;L^1 \cap L^r(\mathbb R^d))} +  N^{-\alpha } +  N^{-\frac{1}{2} +\frac{\alpha d}{r^\prime}   + \varepsilon} \right),
\end{align*}
with   
$C^\prime = C^*\max\!\left(1,\, 4\sqrt{T}d C_{\widetilde{u}_N}\max_{z \in \{1,r\}} C_{\Delta,K,p,z} \, \nu,\, 2C_S\right)$, where $C_S$ is the positive constant given in Proposition \ref{Prop:widetilde_S_NinL} and  $C_{\widetilde{u}_N} := \sup_{N \in \mathbb{N}} \Vert{}\widetilde{u}_N\Vert{}^2_{L^{2m}(\Omega; C([0,T]; L^1 \cap H^{\beta,r}(\mathbb{R}^d)))}$.

\end{proof}

\begin{remark}
    Theorem \ref{Teo:StrongConv} uses the conditions $H_1,H_2,H_3$ of Proposition \ref{Prop:MK_strong_ex_un}   to obtain Lipschitz control of the convolved force in the regularization error. The following corollaries replace this control by a H\"older estimate. In the critical case  $a=d-2$, Corollary \ref{Cor:strong_con_critic} allows $\beta=0$, with H\"older exponent $\zeta=1-d/r$.  In the case $a<d-2$, where $H_1$ already assumes $\beta=0$, Corollary \ref{Cor:strong_con_sub} relaxes the integrability conditions used to obtain Lipschitz continuity. In both cases, the regularization contribution becomes $N^{-\alpha\zeta}$, while the fluctuation exponent remains $1/2-\alpha d/r'$, under the respective assumptions of the corollaries.\end{remark}

\begin{corollary}\label{Cor:strong_con_critic}
   Let $d\ge 3$. Let $K$ be given by \eqref{eq:JL_potential_force}, with $d-2= a>b>0$, and $r$ as in Theorem \ref{Teo: EUMild}. Let $r^\prime$ be its conjugate exponent. Let $u_0 \in L^1 \cap L^{r}(\mathbb R^d)$ be a probability density function, and assume that the initial conditions   \eqref{Initial condition v.a.} and \eqref{Initial condition density}, with $\beta=0$, are satisfied. Fix $T>0$ satisfying \eqref{Lemma contr: Stima1}.
Moreover, let the mesoscopic parameter $\alpha$ satisfy
\begin{equation*}\label{eq:alfaBessel_2}
    0<\alpha<\displaystyle\frac{r^\prime}{2d}.
\end{equation*}

Let the dynamics of the particle system be given by \eqref{SDE:MesoReg}, let $\widetilde{u}_N$ be the empirical density \eqref{MildFormula u2} and $u$ be the mild solution of the PDE \eqref{PDE: FP1}. 
Let $\varrho  =  \min\! \left(\alpha\zeta ,\frac{1}{2}-  \frac{\alpha d}{r^\prime}\right)$  where $\zeta=1-d/r$ is the H\"older continuity exponent given in Proposition \ref{Prop:HoderSpaceSubCri} case \emph{ii)} with $q=r$.

Then, we have that for any $\varepsilon>0$ and any $m\geq 1$,   there exists a constant $C>0$ such that for all $N\in\mathbb{N}$
\begin{align*}
 \left\| \widetilde{u}_N-u \right\|_{L^m(\Omega; C_T^r)} 
&\leq C  \left(\left\|u_N(0,\cdot)- u_0\right\|_{L^m(\Omega; L^1 \cap L^r(\mathbb R^d))} +  N^{-\varrho +\varepsilon}\right),
\end{align*} 
\end{corollary}

\begin{proof}
The proof is analogous to that of Theorem \ref{Teo:StrongConv}. The main difference is that, in order to estimate the function $F(t,x)$ given in \eqref{eq:ForceTerm}, we use the H\"older continuity of $K \ast \widetilde{u}_N$ given by Proposition \ref{Prop:HoderSpaceSubCri} case \emph{ii)} We note that under the conditions $r\ge p^\prime$ and $a=d-2$, we have $r>d$ (see Remark \ref{Re:MK_strong_ex_un}). 
\end{proof}

It is also possible to relax the assumption on $r$ when $d-2>a>b>0$.

\begin{corollary}\label{Cor:strong_con_sub} 
Let $d\ge 3$. Let $K$ be given by \eqref{eq:JL_potential_force}, with $d-2> a>b>0$. Let $p \in (1,d/(a+1))$ and $s\in (1,d/(a+2)]$, and let $p'$ and $s'$ be their conjugate exponents, respectively. Let $r$ be such that $p'\le r< s'$ and $r\ge 2$. Let $u_0 \in L^1 \cap L^{r}(\mathbb R^d)$ be a probability density function, and assume that the initial conditions \eqref{Initial condition v.a.} and \eqref{Initial condition density}, with $\beta=0$, are satisfied. Fix $T>0$ satisfying \eqref{Lemma contr: Stima1}. Let $\alpha$ be such that 
$$
   0<\alpha<\displaystyle\frac{r^\prime}{2d}.
$$
Let the dynamics of the particle system be given by \eqref{SDE:MesoReg}, let $\widetilde{u}_N$ be the empirical density \eqref{MildFormula u2}, and let $u$ be the mild solution of the PDE \eqref{PDE: FP1}. Define $\varrho = \min\! \left(\alpha\zeta, \frac{1}{2} - \frac{\alpha d}{r^\prime}\right)$, where $\zeta:=1-d/q$ is the H\"older continuity exponent for some $q>d$ satisfying $q\in (d/(b+1),+\infty)$, and where the relation among $s$, $r$, and $q$ is $\frac{1}{r} = \frac{1}{s'} + \frac{1}{q}$.

Then, for any $\varepsilon>0$ and any $m\geq 1$, there exists a constant $C>0$ such that for all $N\in\mathbb{N}$,
\begin{align*}
 \left\|\widetilde{u}_N-u  \right\|_{L^m(\Omega; C_T^r)} 
&\leq C  \left(\left\|u_N(0,\cdot)- u_0\right\|_{L^m(\Omega; L^1 \cap L^r(\mathbb R^d))} +  N^{-\varrho +\varepsilon}\right).
\end{align*}  
\end{corollary}
\begin{proof}
The proof is analogous to that of Theorem \ref{Teo:StrongConv}. The main difference is that, in order to estimate the function $F(t,x)$ defined in \eqref{eq:ForceTerm}, we rely on the H\"older continuity of $K \ast \widetilde{u}_N$ provided by Proposition \ref{Prop:HoderSpaceSubCri}, case \emph{i)}, with the parameter $s$ playing the role of $p$ as in Proposition \ref{Prop:Grad_K2-bis}.
\end{proof}

\begin{remark}
When $d \ge 3$ and $a \in (0, d-2)$, Theorem \ref{Teo:StrongConv} and Corollary \ref{Cor:strong_con_sub} can also be established for $r \in (1,2)$. To this end, one first adapts Proposition \ref{Prop:est_u_N_Bessel} by setting $\beta = 0$, restricting the mesoscale parameter to $\alpha \in (0, 1/d)$, and employing Proposition \ref{Prop:widetilde_S_NinL} instead of Proposition \ref{Prop:est_stochastic_convolution_bessel}. Subsequently, let $p, p', s,$ and $s'$ be as in the previous corollary, and assume $r \ge p'$. If $r \ge s'$, the fluctuation error becomes $\rho = \min\!\left(\alpha, \, \frac{1}{2}(1-\alpha d)\right)$. Conversely, if $r < s'$, the fluctuation error becomes $\rho = \min\!\left(\alpha\zeta, \, \frac{1}{2}(1-\alpha d)\right)$. Both parameter regimes are non-empty for suitable choices of $d$, $a$, and $r$ satisfying the required integrability conditions.
\end{remark}

Finally, we may state the convergence result.

\begin{proposition}[Strong convergence of the regularized particle system]\label{prop:convergence_u_N_localised}
  Let the dynamics of the particle system be given by \eqref{SDE:MesoReg}, $\widetilde{u}_N$ the empirical density \eqref{MildFormula u2} and $u$ the mild solution of the PDE \eqref{PDE: FP1}. Under the assumptions of Theorem \ref{Teo:StrongConv}, or Corollary \ref{Cor:strong_con_critic}, or Corollary \ref{Cor:strong_con_sub}  if the empirical density converges to the initial density  $u_0$  in $L^m(\Omega,L^1 \cap L^r(\mathbb R^d))$, i.e.   
    \begin{equation}\label{eq:lim_ini_cond}
    \lim\limits_{N \to+ \infty}  u_N(0,\cdot)   = u_0 \mbox{ in } L^m(\Omega,L^1 \cap L^r(\mathbb R^d)),\end{equation}
    then, 
   $$
 \lim\limits_{N \to+ \infty} \left\| \widetilde{u}_N-u  \right\|_{L^m(\Omega; C_T^r)} 
= 0.
$$
\end{proposition}
\begin{proof}
The conclusion follows easily by taking the limit as $N \to +\infty$ in \eqref{quasi-continuity_mth_moment} and using condition \eqref{eq:lim_ini_cond}.
\end{proof}

\subsection{The SDE at the mesoscale in \texorpdfstring{$\mathbb{R}^d$}{R^d}:   convergence in probability}
 We now use the estimates established for the cut-off particle system to study the empirical density $u_N$ of the original system.
 
 \smallskip
 
 The key observation is that, whenever $\widetilde u_N$ remains sufficiently close to the limiting density $u$, the corresponding force field stays below the cut-off threshold throughout $[0,T]$. A stopping-time argument and pathwise uniqueness then show that the two particle systems, driven by the same initial positions and Brownian motions, coincide on this event. Theorem \ref{Teo:StrongConv} ensures that the probability of this event tends to one as $N\to+\infty$. This allows us to transfer the estimates for $\widetilde u_N$ into probability bounds for $u_N$. 
 
 \smallskip

 The following proposition establishes convergence in probability for the original empirical density.
\begin{proposition}\label{Prop:weak_con}
    Let us suppose that all the  hypotheses of Theorem \ref{Teo:StrongConv} are satisfied, with $\varrho  =  \min\! \left(\alpha ,\frac{1}{2}-  \frac{\alpha d}{r^\prime}  \right)$. Then, under the initial condition \eqref{eq:lim_ini_cond}, the empirical density  $u_N$ defined in \eqref{eq:def_empirical_density} converges in probability to the mild solution $u$   of the PDE \eqref{PDE: FP1}  in $C([0,T];L^1 \cap L^r(\mathbb R^d))$, i.e., for any $\eta>0$
    \begin{equation} \label{eq:convergence_in_probability}
\lim\limits_{N\rightarrow +\infty} \mathbb P\left(\left\| u_N-u \right\|_{ C_T^r}\ge \eta  \right)=0.
\end{equation}
 
\end{proposition}
\begin{proof}

We obtain the claim \eqref{eq:convergence_in_probability} by proving that for any $m\ge 1$ and any $\varepsilon \in (0,\varrho)$, there exists a constant $C>0$ such that for all $N\in\mathbb{N}$,
$$ 
 \mathbb P\left(\left\| u_N-u \right\|_{ C_T^r}\ge \eta  \right)
\le \frac{C}{\eta^m }  \left(\left\|u_N(0,\cdot)- u_0\right\|_{L^m(\Omega; L^1 \cap L^r(\mathbb R^d))} +  N^{-\varrho +\varepsilon}\right)^m.
$$ 

Fix $\eta>0$ and define
$$
D_{N,\eta}
 :=
 \left\{\omega\in \Omega: \|\widetilde u_N(\cdot,\cdot,\omega)-u(\cdot,\cdot)\|_{C_T^r}<\eta\right\}.
$$
Under the assumptions of Theorem \ref{Teo:StrongConv}, let $E_N=C^\prime  \left(\left\|u_N(0,\cdot)- u_0\right\|_{L^m(\Omega; L^1 \cap L^r(\mathbb R^d))} +  N^{-\varrho +\varepsilon}\right)$. By Markov's inequality, 
$$
\mathbb{P}\left( \|\widetilde u_N - u\|_{C_T^r} \ge \eta \right) \le \frac{\mathbb{E}\left[ \|\widetilde u_N - u\|_{C_T^r}^m \right]}{\eta^m} \le \frac{E_N^m}{\eta^m}.
$$
Therefore, 
$$
\mathbb{P}(D_{N,\eta})  \ge 1 - \left(\frac{E_N}{\eta}\right)^m.
$$
Hence, for any $\eta>0$ and $\delta>0$, there exists $N_0\in \mathbb N$ such that for all $N\ge N_0$, $\mathbb{P}(D_{N,\eta}) \ge 1-\delta$.

By Proposition \ref{prop:LJ_con_op}, on the event $D_{N,\eta}$, for every $t\in[0,T]$ and $x\in\mathbb R^d$, we have
$$
\max_{1\le i\le d}
\left|(K*\widetilde u_N)_i(t,x)\right|
 \le C_{1,\nu}C_{K,p,q}\left(\|u\|_{C_T^r} +\|\widetilde u_N-u\|_{C_T^r}\right) < B_{\eta,u},
$$
where $B_{\eta,u}$ is defined in \eqref{Eq:threshold}.

Let us set
$$
\tau_N := \inf\left\{t\in[0,T]: \max_{1\le i\le d} \|(K*\widetilde u_N)_i(t,\cdot)\|_{L^\infty(\mathbb R^d)} \ge B_{\eta,u}\right\},
$$
with the convention that the infimum of the empty set is $+\infty$.

The continuity in time of $\widetilde u_N$, together with Proposition \ref{prop:LJ_con_op}, ensures that $\tau_{N}$ is a stopping time. Since both systems are driven by the same initial positions and Brownian motions, pathwise uniqueness yields
$$
X^i_{t\wedge\tau_N} =\widetilde X^i_{t\wedge\tau_N}, \qquad t\in[0,T],\quad i=1,\ldots,N,
$$
almost surely. On $D_{N,\eta}$, the preceding strict bound gives $\tau_N=+\infty$. Consequently, on $D_{N,\eta}$ we have $t \wedge \tau_N = t$, yielding,
$$
X^i_t=\widetilde X^i_t \quad\text{for every }t\in[0,T]\text{ and }i=1,\ldots,N.
$$
Hence, $u_N=\widetilde u_N$ on $D_{N,\eta}$,
up to a null set.  

It follows that
$$
\left\{\|u_N-u\|_{C_T^r}\ge\eta\right\}
 \subseteq D_{N,\eta}^{\,c}
 =
 \left\{\|\widetilde u_N-u\|_{C_T^r}\ge\eta\right\},
$$
up to a null set. Therefore, for $m\ge1$, Markov's inequality and Theorem \ref{Teo:StrongConv} give
$$
\mathbb P\!\left(\|u_N-u\|_{C_T^r}\ge\eta\right)
\le
 \mathbb P\!\left(
   \|\widetilde u_N-u\|_{C_T^r}\ge\eta
 \right) \le \frac{\mathbb{E}\left[ \|\widetilde u_N - u\|_{C_T^r}^m \right]}{\eta^m} \le \frac{E_N^m} {\eta^m} \xrightarrow{N\to + \infty} 0.
$$
\end{proof}

\begin{remark}
 Following the same line of argument as in Corollary \ref{Cor:strong_con_critic} and Corollary \ref{Cor:strong_con_sub}, it is possible to relax the hypotheses of Proposition \ref{Prop:weak_con}.
\end{remark}
 
\begin{remark}
 For fixed admissible exponents \(r\) and \(\beta\), Theorem \ref{Teo:StrongConv} 
provides the estimate \[
    \|\widetilde u_N-u\|_{L^m(\Omega;C_T^r)}
    \leq C\left(A_N+N^{-\rho(\alpha)+\varepsilon}\right),
\] where $A_N:= \|u_N(0,\cdot)-u_0\|_{L^m(\Omega;L^1\cap L^r(\mathbb{R}^d))}$ and    $\rho(\alpha) := \min\!\left( \alpha,\, 1/2 - \alpha d/r'   \right).$  The exponent \(\rho(\alpha)\) is maximized when these two 
terms coincide, namely at $\alpha_*= {1}/{(2(1+d/r'))}, $ $\rho_*=\rho(\alpha_*) ={1}/{(2(1+d/r'))}.$
Since \(0\le \beta<1\), this choice satisfies $ \alpha_*< 1/(2(\beta+d/r')), $ and is therefore admissible under condition \eqref{eq:alfaBessel}. For this choice, the estimate becomes
$$
    \|\widetilde u_N-u\|_{L^m(\Omega;C_T^r)}
    \leq C\left(A_N+N^{-\rho_*+\varepsilon}\right).
$$
This choice optimizes the explicit dynamical exponent in the bound for fixed \(r\) and \(\beta\). The overall convergence rate also depends on \(A_N\).
In particular, the rate \(N^{-\rho_*+\varepsilon}\) for the full error requires \(A_N=O(N^{-\rho_*+\varepsilon})\). No sharpness of this rate beyond the present estimates is claimed.
\end{remark}


\appendix

\section{Proofs in Section \ref{Sec:LJ pot}}\label{sec:appendix_A}

\subsection{Proof of Proposition \ref{Prop:K_L1}}\label{appendix:proof:prop_Prop:K_L1}

  We have to show that for any compact set $D \subset \mathbb{R}^d$, 
    \begin{equation*}
        \int_D \left|K(x)\right| \ dx=\int_D\epsilon\left|\frac{R_0^a}{|x|^{a+1}}-\frac{R_0^b}{|x|^{b+1}}\right| \ dx<+\infty.
    \end{equation*}

    If $0 \notin D$, the integral is trivially finite,  since the function is continuous on a compact set.
Now, let  $ 0 \in D$. Fix $\delta>0$ such that $D \subset B(0,\delta) $, where $B(0,\delta)$ denotes the ball of radius $\delta $ centred at the origin. 
 By considering  the spherical coordinate system $r>0$, $\theta_i \in [0,\pi]$ for all $i \in \{1,\dots,d-2\}$, and $\phi \in [0,2\pi)$, such that 
$
     x_1= r\cos(\theta_1), 
   x_2= r\sin(\theta_1)\cos(\theta_2), \cdots,   
  x_d= r \sin(\theta_1)\sin(\theta_2)...\sin(\theta_{d-2})\sin(\phi)  
$
and given the determinant of the Jacobian matrix of the transformation $
    \text{det}\left(J_d\right) = r^{d-1} \sin^{d-2}(\theta_1) \sin^{d-3}(\theta_2)...\sin(\theta_{d-2}). 
$ We obtain
  \begin{align*}
        \int_D \left|K(x)\right| \ dx &<\epsilon\int_{B(0,\delta)}  \frac{R_0^a}{|x|^{a+1}}+\frac{R_0^b}{|x|^{b+1}} \ dx\\
        &=\epsilon\int_0^\pi\int_0^\pi  \cdots\int_0^{2\pi} \int_0^{\delta} \left(\frac{R_0^a}{r^{a+1}}+\frac{R_0^b}{r^{b+1}}\right) \text{det}\left(J_d\right) \ d\theta_1d\theta_2 \dots d\phi dr \\ 
        &=\epsilon\int_0^\pi\sin^{d-2}(\theta_1) \ d\theta_1\int_0^\pi \sin^{d-3}(\theta_2) 
 \ d\theta_{2}  \cdots\int_0^{2\pi}d\phi \int_0^{ \delta} \left(\frac{R_0^a}{r^{a+1}}+\frac{R_0^b}{r^{b+1}}\right)r^{d-1} \ dr \\
       &=  C_\epsilon\int_0^\delta \displaystyle \frac{R_0^a}{r^{a+2-d}} \ dr+   C_\epsilon \int_0^\delta \displaystyle \frac{R_0^b}{r^{b+2-d}} \ dr, 
    \end{align*}
where $C_\epsilon$  depends upon $\epsilon $ and $d$.
    The first integral is finite when $a+2-d<1$, hence, $a<d-1$. Similarly, $b<d-1$.  Then, the local integrability property  is proven.

\subsection{Proof of Proposition \ref{Prop:integrability_pq}}\label{appendix:Prop:integrability_pq}

Let  us first consider $R_0\le  \nu$. 
By applying the spherical coordinate change, and  observing that   $R_0^br^{-(b+1)}>R_0^ar^{-(a+1)}$ if $r>R_0$, we have 
\begin{equation*}
\begin{split}
\displaystyle\int_{B(0,\nu)} \left|K(x)\right|^p \ dx =&\, \epsilon^p\int_{B(0,\nu)}\left|R_0^a|x|^{-(a+1)}-R_0^b|x|^{-(b+1)}\right|^p  \ dx\\
 = &\,C_\epsilon\int_0^\nu  \left|R_0^ar^{-(a+1)}-R_0^br^{-(b+1)}\right|^p  \ r^{d-1} \ dr \\
= &\,C_\epsilon\displaystyle\int_0^{R_0} \left(R_0^ar^{-(a+1)}-R_0^br^{-(b+1)}\right)^p \ r^{d-1} \ dr\\&+C_\epsilon\int_{R_0}^\nu \left(R_0^br^{-(b+1)}-R_0^ar^{-(a+1)}\right)^p \  r^{d-1} \ dr\\
 =&\, I_1+I_2.
\end{split}
\end{equation*}
The integral term $I_2$ is  finite, because the integrands is continuous on a compact set.
 
As for $I_1$, if $g(r)=R_0^ar^{-(a+1)}-R_0^br^{-(b+1)}$, then $     g(r) \sim_0 R_0^ar^{-(a+1)},$ hence 
\begin{equation*}
 I_1  \sim_0   C_\epsilon\displaystyle\int_0^{R_0} R_0^{ap}r^{-(a+1)p+d-1} \ dr,
\end{equation*}
which  is finite when $(a+1)p-d+1<1$, that is $p \in \left[1,  {d}/({a+1})\right)$, under the condition $a<d-1$. Hence, property \emph{i)} is proven by a similar calculation in the case $R_0>\nu$.  
 Property \emph{ii)}  may be achieved in a similar way.

\subsection{Proof of Proposition \ref{Prop:Grad_K2}}\label{appendix:Prop:Grad_K2}
 
Let us note that the distributional
Jacobian of $K$ is 
\begin{equation}\label{eq:Nabla_K}
\nabla K   =     \nabla_* K   + c_K I_d\delta_0,
    \qquad    c_K:=\frac{\epsilon R_0^{d-2}|\mathbb S^{d-1}|}{d}.
    \end{equation} 
Here, $\nabla_*K$ is understood in the principal value sense:
$$
\langle \nabla_*K,\varphi\rangle := \lim_{\rho\downarrow 0}
\int_{|x|>\rho}\nabla_*K(x)\varphi(x)\,dx, \qquad \varphi\in C_c^\infty(\mathbb R^d).
$$
For $x\ne0$, $\nabla_*K(x)$ denotes the classical Jacobian
of $K$; from \eqref{eq:HessianMatrixK} and since $a=d-2$, it follows that
$$     
\nabla_*K(x) =\frac{\epsilon R_0^{d-2}}{|x|^d} \left(I_d-d\frac{x\otimes x}{|x|^2}\right) - \frac{\epsilon R_0^b}{|x|^{b+2}} \left(I_d-(b+2)\frac{x\otimes x}{|x|^2}\right), \qquad x\ne 0.
$$
Fix $\nu>0$, and choose a radial function $\chi\in C_c^\infty(\mathbb R^d)$ such that  $\chi=1$ on $B(0,\nu)$, and $\chi=0$ on $B^c(0,2\nu)$. Using this cut-off function, we decompose the classical Jacobian into a part supported near the singularity and a part supported away from it
\begin{equation}\label{eq:nabla*K_decomposition} 
\nabla_*K(x)   =   \chi  \nabla_*K  +(1-\chi) \nabla_*K    =    B_{NS} + B_{AS}.
\end{equation}
As in Proposition \ref{Prop:integrability_pq}, the condition $\bar q(b+2)>d$, together with the fact that $b<a=d-2$, implies
   $ B_{AS}
    \in L^{\bar q}(\mathbb R^d;\mathbb R^{d\times d}).$
Hence, for any $f\in L^1\cap L^{\bar{q}}(\mathbb R^d)$, Young's convolution inequality yields
\begin{equation}\label{eq:nabla*K_AS}
\|B_{AS}*f\|_{L^{\bar q}(\mathbb R^d)} \le \|B_{AS}\|_{L^{\bar q}(\mathbb R^d)} \|f\|_{L^1(\mathbb R^d)}.\end{equation}
Near the singularity, the operator admits a further decomposition into a critical component and a remainder $\mathcal R\in L^1(\mathbb R^d;\mathbb R^{d\times d}),$ that is 
\begin{equation}\label{eq:B_NS}
\begin{split}
B_{NS}&=\frac{\epsilon R_0^{d-2}}{|x|^d}\left(I_d-d\frac{x\otimes x}{|x|^2}\right)\chi(x)  -\frac{\epsilon R_0^b}{|x|^{b+2}}
\left(I_d-(b+2)\frac{x\otimes x}{|x|^2}\right) \chi(x)\\
& = : \frac{\xi(x/|x|)}{|x|^d}\chi(x) +\mathcal{R}(x)  =h(x) +\mathcal{R}(x).
\end{split}
\end{equation}
The zero spherical mean of $\xi$ and the radiality
of $\chi$ give 
\[
    \int_{R_1<|x|<R_2}h(x)\,dx=0,
    \qquad 0<R_1<R_2.
\]
Moreover, $|h(x)|+|x||\nabla h(x)|\leq C|x|^{-d}$.
Thus $h$ satisfies the size, H\"ormander smoothness,
and cancellation conditions. By
\cite[Theorems 5.4.1 and 5.4.5]{Grafakos_2014}, applied componentwise,
the principal-value convolution operator associated
with $h$ is bounded on $L^{\bar q}(\mathbb R^d)$,
since $1<\bar q<+\infty$. In what follows, $h*f$ denotes the principal-value operator, and
$B_{NS}*f:=h*f+\mathcal R*f$. The identities are first understood for smooth compactly supported functions and then extended to $L^1\cap L^{\bar q}(\mathbb R^d)$ by density and the bounds established below.
Hence from \eqref{eq:B_NS}, for $f\in L^1\cap L^{\bar q}(\mathbb R^d)$ 
\begin{equation}\label{eq:B_NS_estimate} 
  \|B_{NS} *f\|_{L^{\bar q}(\mathbb R^d)} \le \left\|h  \ast f\right\|_{L^{\bar q}(\mathbb R^d)}
       +\|\mathcal R*f\|_{L^{\bar q}(\mathbb R^d)} \le C\|f\|_{L^{\bar q}(\mathbb R^d)}
       +\|\mathcal R\|_{L^1(\mathbb R^d)}\|f\|_{L^{\bar q}(\mathbb R^d)}.
\end{equation}
Finally, from \eqref{eq:Nabla_K}, \eqref{eq:nabla*K_decomposition}, \eqref{eq:nabla*K_AS} and \eqref{eq:B_NS_estimate} the final estimate yields, since
$$
    \|\nabla K*f\|_{L^{\bar q}(\mathbb R^d)}
    \le
       \left\|B_{NS} \ast f\right\|_{L^{\bar q}(\mathbb R^d)}
       +\|B_{AS}*f\|_{L^{\bar q}(\mathbb R^d)}
       +c_K\|f\|_{L^{\bar q}(\mathbb R^d)}\\
    \le C\|f\|_{L^1 \cap L^{\bar q}(\mathbb R^d)}.
$$

\subsection{Proof of Proposition \ref{Prop:Grad_K3}}\label{appendix:Prop:Grad_K3}

We start by noticing that under the assumption $\eta := \beta - \displaystyle\frac{d}{r} \in (2-d+a, 1)$, the Bessel potential space $H^{\beta,r}(\mathbb{R}^d)$ is continuously embedded into the H\"older space $C^{\beta - d/r}(\mathbb{R}^d)$ (see, for example, \cite{Triebel_1978}).

Let $\nabla_* K(x)$ denote the classical Jacobian of $K$ on $\mathbb{R}^d \setminus \{0\}$ given in \eqref{eq:HessianMatrixK}, and fix $\nu > 0$. 

We first determine the explicit representation of the distributional Jacobian $\nabla K$ acting on a test function $\psi \in C_c^\infty(\mathbb{R}^d)$. 
Since $K \in L^1_{\mathrm{loc}}(\mathbb{R}^d)$,  by Proposition \ref{Prop:K_L1}, we have
$$
\langle \nabla K, \psi \rangle = -\int_{\mathbb{R}^d} K(y) \otimes \nabla \psi(y)\,dy = -\lim_{\varepsilon \to 0^+} \int_{|y| > \varepsilon} K(y) \otimes \nabla \psi(y)\,dy.
$$
For any $0 < \varepsilon < \nu$, we split the integration domain into $\{\varepsilon < |y| < \nu\}$ and $\{|y| \ge \nu\}$. Rewriting $\nabla \psi(y) = \nabla\big(\psi(y) - \psi(0)\big)$ on the annulus $\{\varepsilon < |y| < \nu\}$ and integrating by parts yields
\begin{align*}
-\int_{|y| > \varepsilon} K(y) \otimes \nabla \psi(y)\,dy 
&= \int_{\varepsilon < |y| < \nu} \nabla_* K(y)\big(\psi(y) - \psi(0)\big)\,dy + \int_{|y| \ge \nu} \nabla_* K(y)\psi(y)\,dy \\
&  + \left(\int_{|y|=\nu} K(y) \otimes \frac{y}{|y|}\,d\sigma(y)\right)\psi(0) + \int_{|y|=\varepsilon} K(y) \otimes \frac{y}{|y|}\big(\psi(y) - \psi(0)\big)\,d\sigma(y).
\end{align*}

We start by addressing the boundary integrals. 

On $\{|y|=\nu\}$, by using the radial structure $K(y) = K^*(|y|)\frac{y}{|y|}$, we have 
$$
\int_{|y|=\nu} K(y) \otimes \frac{y}{|y|}\,d\sigma(y) = c_{K,\nu} I_d, \qquad \text{with } c_{K,\nu} := \frac{|\mathbb{S}^{d-1}|}{d}\nu^{d-1} K^*(\nu).
$$
On $\{|y|=\varepsilon\}$, since $|\psi(y) - \psi(0)| \le |\psi|_\eta \varepsilon^\eta$ and $|K(y)| \le \epsilon C_{1,\nu}\varepsilon^{-(a+1)}$ for $\varepsilon < \nu$, we have
$$
\left|\int_{|y|=\varepsilon} K(y) \otimes \frac{y}{|y|}\big(\psi(y) - \psi(0)\big)\,d\sigma(y)\right| \le \epsilon C_{1,\nu}|\mathbb{S}^{d-1}|\,|\psi|_\eta\, \varepsilon^{d-2-a+\eta} \xrightarrow{\varepsilon \to 0^+} 0.
$$
Furthermore, the pointwise estimate \eqref{eq:stima_gradK}
for the classical Jacobian $\nabla_*K(y)$ remains valid
for $0<|y|<\nu$, also in the regime $d-2<a<d-1$,
as follows directly by differentiating $K$ away from the origin.
Consequently,
$$
\big|\nabla_* K(y)\big(\psi(y) - \psi(0)\big)\big| \le C_{2,\nu}\epsilon |\psi|_\eta |y|^{\eta - (a+2)} \in L^1(B(0,\nu)),
$$
since $\eta > 2-d+a$. Passing to the limit as $\varepsilon \to 0^+$ by Lebesgue's dominated convergence theorem, and extending the relation by density to $f \in L^1(\mathbb{R}^d) \cap H^{\beta,r}(\mathbb{R}^d)$, we obtain that for every $x \in \mathbb{R}^d$,
\begin{equation}\label{eq:nablaK_conv_rep}
\nabla K \ast f(x) = \int_{|y| < \nu} \nabla_* K(y)\big(f(x-y) - f(x)\big)\,dy + \int_{|y| \ge \nu} \nabla_* K(y)f(x-y)\,dy + c_{K,\nu} f(x) I_d.
\end{equation}

We now estimate each term on the right-hand side of \eqref{eq:nablaK_conv_rep} for any $f \in L^1 \cap H^{\beta,r}(\mathbb{R}^d)$.

 \smallskip
\emph{Near Singularity:} since $|f(x-y) - f(x)| \le |f|_\eta |y|^\eta$ and $\eta > 2 - d + a$, switching to spherical coordinates gives
\begin{align*}
\left|\int_{|y| < \nu} \nabla_* K(y)\big(f(x-y) - f(x)\big)\,dy\right| 
& \le C_{2,\nu}\epsilon |f|_\eta \int_{B(0,\nu)} |y|^{\eta - (a+2)}\,dy 
 = \frac{C_{2,\nu}\epsilon |\mathbb{S}^{d-1}|\nu^{\eta - (a+2-d)}}{\eta - (a+2-d)}\,|f|_\eta.
\end{align*}
\emph{Away from Singularity:} since $|\nabla_* K(y)| \le C_{2,\nu} J_2(y) \le C_{2,\nu}\epsilon \nu^{-(b+2)}$, H\"older's inequality yields
$$
\left|\int_{|y| \ge \nu} \nabla_* K(y)f(x-y)\,dy\right| 
\le C_{2,\nu}\|J_2\|_{L^\infty(\mathbb{R}^d)}\|f\|_{L^1(\mathbb{R}^d)} 
= C_{2,\nu}\epsilon \nu^{-(b+2)}\|f\|_{L^1(\mathbb{R}^d)}.
$$
\emph{The boundary Term.} 
$$
|c_{K,\nu} f(x) I_d| \le |c_{K,\nu}|\,\|f\|_{L^\infty(\mathbb{R}^d)}.
$$ 
Combining the three estimates above, we deduce that for all $x \in \mathbb{R}^d$,
$$
|(\nabla K \ast f)(x)| \le C\left(\|f\|_{L^1(\mathbb{R}^d)} + \|f\|_{L^\infty(\mathbb{R}^d)} + |f|_\eta\right) \le C\|f\|_{L^1 \cap H^{\beta,r}(\mathbb{R}^d)},
$$
where $C = C(\epsilon, R_0, \nu, a, b, d, \beta, r) > 0$. Then, the conclusion follows.

\smallskip
\begin{remark}
    In the previous proof, choosing $\nu = R_0$ yields $K^*(R_0) = 0$ and thus $c_{K,R_0} = 0$.
\end{remark}

 \section{Proof of Proposition \ref{Prop:KN_bounded}}\label{appendix:KN_properties}
  
       Fix $N \in \mathbb{N}$. Let $\kappa=\operatorname{supp}(V)$. Then  $\kappa_N:=N^{-\alpha}\kappa $ is the support of the mollifier $V_N$.  By Proposition \ref{Prop:K_L1},  $K\in L^1_{loc}(\mathbb R^d)$. Since $V_N$ has compact support, $K \ast V_N(x)$ is well defined  for any $x\in \mathbb R^d$. Indeed, by H\"older's inequality for the conjugate pair $L^1$–$L^\infty$ we get, for any $x\in \mathbb R^d$,
    \begin{align*}
    |K \ast V_N(x)| & \le  
    \int_{\kappa_N}|K(x-y)||V_N(y)|  \ dy  \le 
    \|K\|_{L^1(x-\kappa_N)}\|V_N\|_{L^\infty(\mathbb{R}^d)},  \end{align*} 
    where 
    $x - \kappa_N := \{\, x - z : z \in \kappa_N \,\}.$ We can prove a uniform (in $x$) bound for any translation of the compact set $\kappa_N$, that is \begin{equation}\label{eq:local_uniform_bound_K_on_traslation}\sup_{x \in \mathbb{R}^d} \|K\|_{L^1(x-\kappa_N)}=C^*<+\infty, \end{equation}
    From  the uniform bound of local averages \eqref{eq:local_uniform_bound_K_on_traslation} we would get the desired regularity of the convolution $K_N$;  for any fixed $N\in \mathbb N,$
    $$\|K_N\|_{L^\infty(\mathbb{R}^d) } = \|K \ast V_N\|_{L^\infty(\mathbb{R}^d) } \le \left(\sup_{x \in \mathbb{R}^d}\|K\|_{L^1(x-\kappa_N)}\right)\|V_N\|_{L^\infty(\mathbb{R}^d)}\le C^* N^{d\alpha}\|V\|_{L^\infty(\mathbb{R}^d)}<+\infty.$$
      By classical results \cite[Theorem 1.3.1]{Hormander_2003}, we also have that $V_N \ast K \in C^2(\mathbb{R}^d)$. 
In conclusion, $K \ast V_N$ is well defined. Thus $K_N=K\ast V_N=V_N\ast K \in C^2\cap L^\infty(\mathbb{R}^d)$, which proves the claim.

 It remains to prove the uniform bound on local averages in \eqref{eq:local_uniform_bound_K_on_traslation}. 
We first notice that
$$|K(y)|\le
\begin{cases}
\displaystyle \frac{C}{|y|^{a+1}}, & \text{if } |y| \le R_0 \\
\displaystyle \frac{C}{|y|^{b+1}}, & \text{if } |y| > R_0
\end{cases}
$$
for some constant $C > 0$. Let $\delta>0$ such that $\kappa_N\subseteq B(0,\delta)$, so that $x-\kappa_N \subseteq B(x,\delta)$; hence,  we estimate the local integral of $K$ over the shifted support $x-\kappa_N$ as follows

$$\int_{x-\kappa_N}|K(y)| \, dy \le \int_{B(x,\delta)} |K(y)| \, dy.$$
We distinguish two cases. Let us first suppose  $ |x| > R_0 + \delta $. 
Then, for all $y \in B(x, \delta) $, we have $|y| > R_0 $, so that by writing $y=x+u$ with $|u|\le \delta$, 
$$\int_{x-\kappa_N}|K(y)| \, dy \le \int_{B(x,\delta)} |K(y)| \, dy\le C\int_{B(x, \delta)} \frac{1}{|y|^{b+1}}\,dy = C\int_{B(0, \delta)} \frac{1}{|x + u|^{b+1}}\,du\le \frac{C|B(0, \delta)|}{(|x| - \delta)^{b+1}}.$$

 Now, let us consider the case $|x| \le R_0 + \delta$. In this case, the integration domain $ x - \kappa_N \subset B(x, \delta) \subset B(0, R_0 + 2\delta)$ is contained in a fixed compact set. Since $K \in L^1_{\text{loc}}(\mathbb{R}^d) $, there exists a finite constant $C^\prime$ such that
$$
\sup_{|x| \le R_0 + \delta} \int_{x - \kappa_N} |K(y)|\,dy =C^\prime.
$$

Combining the two cases,  we obtain a uniform bound
$$
\sup_{x\in \mathbb{R}^d} \int_{x-\kappa_N} |K(y)| \, dy
\le \max\left\{ C', \; \frac{C\, |B(0,\delta)|}{R_0^{b+1}} \right\} <+ \infty.
$$

  \section{Proofs related to the stochastic convolution integrals}\label{Est:S_N}

    \subsection{Proof of Proposition \ref{Prop:est_stochastic_convolution_bessel}}
 
To prove this estimate, we will first show that for any sufficiently large moment $m_0 \ge 1$ and a suitably chosen $\theta \in (0, 1-\beta]$, 
$$ \|S_N(t,\cdot)-S_N(s,\cdot)\|_{L^{m_0}(\Omega;H^{\beta,z}(\mathbb R^d))} \le N^{-1/2+\alpha(d+\beta - d/z+2\theta)}C^*(t-s)^{\theta/2}, \quad \text{for } 0 \le s \le t \le T.$$ 
The claim then follows from Corollary 4.4 in \cite{2021_Richard}, in conjunction with the Garsia-Rodemich-Rumsey continuity lemma for Banach spaces \cite{2010_Friz}, applied to the sufficiently large moment $m_0$, and subsequently deduced for the original moment $m \ge 1$ via H\"older's inequality.

We begin by noticing that, fixing $N \in \mathbb{N}$ and $0 \le s \le t \le T$, we can decompose the increment without absolute values as
\begin{align*}
S_N(t,x)-S_N(s,x) &= -\frac{\sqrt{2}}{N}\sum_{i=1}^N \int_s^t \nabla e^{(t-u)\Delta}V_N(x-X_u^i) \, dW_u^i \\
&\quad - \frac{\sqrt{2}}{N}\sum_{i=1}^N \int_0^s \nabla e^{(s-u)\Delta}\left(e^{(t-s)\Delta}V_N(x-X_u^i) - V_N(x-X_u^i)\right)\, dW_u^i \\
&= I_1(s,t,x) + I_2(0,s,x),
\end{align*}
where, keeping $s$ and $t$ fixed in the heat semigroup operators, we define for any $r_1 \in [s,t]$ and $r_2 \in [0,s]$ the function-valued processes
\begin{align*}
I_1(s,r_1,x) &:= -\frac{\sqrt{2}}{N}\sum_{i=1}^N \int_s^{r_1} \nabla e^{(t-u)\Delta}V_N(x-X_u^i) \, dW_u^i, \\
I_2(0,r_2,x) &:= -\frac{\sqrt{2}}{N}\sum_{i=1}^N \int_0^{r_2} \nabla e^{(s-u)\Delta}\left(e^{(t-s)\Delta}V_N(x-X_u^i) - V_N(x-X_u^i)\right)\, dW_u^i.
\end{align*}
While $I_1(s,t,\cdot)$ and $I_2(0,s,\cdot)$ are not martingales in $t$ and $s$ due to the dependence on $t$ and $s$ inside the heat kernels, the processes $r_1 \mapsto I_1(s,r_1,\cdot)$ and $r_2 \mapsto I_2(0,r_2,\cdot)$ are martingales on $[s,t]$ and $[0,s]$, respectively.

We begin by estimating the quantity 
$$\|I_1(s,r_1,\cdot)\|^{m_0}_{L^{m_0}(\Omega; H^{\beta,z}(\mathbb R^d))}=\mathbb{E}\left[ \left\| \frac{\sqrt{2}}{N}\sum_{i=1}^N  \int_s^{r_1} \left(\mathrm{I}-\Delta \right)^{\frac{\beta}{2}} \nabla e^{(t-u)\Delta}V_N(\cdot-X_u^i) \ dW_u^i \right\|^{m_0}_{L^z(\mathbb R^d)}\right].$$ 

Since $z \ge 2$, the space $L^z(\mathbb R^d)$ is a UMD Banach function space. By the Burkholder-Davis-Gundy inequality in UMD spaces and the square function criterion (see Corollary 3.11 in \cite{2007_vanNeerven_veraar_weis}), the $L^{m_0}$-moment of the stochastic integral is bounded by the $L^{m_0}$-moment of the $L^z(\mathbb R^d)$-norm of the square function. Furthermore, since $z/2 \ge 1$, Minkowski's integral inequality allows us to bound the square function by the time integral of the squared $L^z(\mathbb R^d)$-norms. Hence,
\begin{align*}
\|I_1(s,r_1,\cdot)\|^{m_0}_{L^{m_0}(\Omega; H^{\beta,z}(\mathbb R^d))} \le C_{m_0,z} \, \mathbb{E} \left[ \left( \frac{1}{N^2} \sum_{i=1}^N \int_s^{r_1} \left\| (\mathrm{I} - \Delta)^{\frac{\beta}{2}} \nabla e^{(t-u)\Delta} V_N \left(\cdot-X_u^i\right) \right\|^2_{L^z(\mathbb R^d)} \, du \right)^{m_0/2} \right]. 
\end{align*}

Now we estimate the integrand of the expected value. For any $\delta \in (0, 1-\beta]$, 
\begin{align*}
& \frac{1}{N^2} \sum_{i=1}^N  \int_s^{r_1} \left\| (\mathrm{I} - \Delta)^{\frac{\beta}{2}} \nabla e^{(t-u)\Delta} V_N  \left(\cdot-X_u^i\right) \right\|^2_{L^z(\mathbb R^d)} \ du  \\ 
= &\frac{1}{N^2} \sum_{i=1}^N  \int_s^{r_1} \left\| (\mathrm{I} - \Delta)^{-\frac{\delta}{2}} \nabla e^{(t-u)\Delta} (\mathrm{I} - \Delta)^{\frac{\beta+\delta}{2}} V_N  \left(\cdot-X_u^i\right) \right\|^2_{L^z(\mathbb R^d)} \ du \\ 
\le & \left(C_\Delta^\prime\right)^2 \, \frac{1}{N}\int_s^{r_1}(t-u)^{\delta-1} \|V_N\|^2_{H^{\beta+\delta,z}(\mathbb R^d)} \ du= \left(C_\Delta^\prime\right)^2 \, \frac{1}{N}\frac{(t-s)^{\delta} - (t-r_1)^{\delta}}{\delta}\|V_N\|^2_{H^{\beta+\delta,z}(\mathbb R^d)}. 
\end{align*}

To estimate $\|V_N\|_{H^{\sigma,z}(\mathbb R^d)}$ for any $\sigma \in [0,2]$, we observe that by the change of variables $y = N^\alpha x$ and the fact that $V \in C^2_c(\mathbb R^d)$, we have at the endpoints $\sigma = 0$ and $\sigma = 2$ by recalling \eqref{eq:VN_Lp}
$$\|V_N\|_{L^z(\mathbb R^d)} = N^{\alpha(d - d/z)} \|V\|_{L^z(\mathbb R^d)},$$
and
\begin{align*}
\|V_N\|_{H^{2,z}(\mathbb R^d)} &= N^{d\alpha} \left\| (\mathrm{I} - \Delta)\left(V(N^\alpha \cdot)\right) \right\|_{L^z(\mathbb R^d)} = N^{\alpha(d - d/z)} \|V - N^{2\alpha}\Delta V\|_{L^z(\mathbb R^d)} \\
& \le N^{\alpha(d - d/z + 2)} \|V\|_{H^{2,z}(\mathbb R^d)}.
\end{align*}
By the complex interpolation identity $H^{\sigma,z}(\mathbb R^d) = [L^z(\mathbb R^d), H^{2,z}(\mathbb R^d)]_{\sigma/2}$ for Bessel potential spaces (see Theorem 1 Section 2.4.2 in \cite{Triebel_1978}), it follows that for any $\sigma \in [0,2]$,
\begin{align}\label{Est: scaling_VN}
\|V_N\|_{H^{\sigma,z}(\mathbb R^d)} \le C_{\sigma,z} \|V_N\|_{L^z(\mathbb R^d)}^{1-\sigma/2} \|V_N\|_{H^{2,z}(\mathbb R^d)}^{\sigma/2} \le C_{\sigma,z} N^{\alpha(d - d/z + \sigma)} \|V\|_{H^{2,z}(\mathbb R^d)}.
\end{align}

Applying \eqref{Est: scaling_VN} with $\sigma = \beta + \delta < 2$ and setting $C_1 = C_\Delta^{\prime} C_{m_0,z}^{1/m_0} C_{\beta+\delta,z} \, \delta^{-1/2}\|V\|_{H^{2,z}(\mathbb R^d)}$, by letting $r_1 \to t$ we obtain
\begin{equation}\label{Est : I_1}
\|I_1(s,t,\cdot)\|_{L^{m_0}(\Omega; H^{\beta,z}(\mathbb R^d))} \le N^{-1/2+\alpha(d+\beta-d/z+\delta)}C_1(t-s)^{\delta/2}. 
\end{equation}

In a similar way, for any $r_2 \in [0,s)$, it follows that
\begin{align*}
&\|I_2(0,r_2,\cdot)\|^{m_0}_{L^{m_0}(\Omega; H^{\beta,z}(\mathbb R^d))} \\ 
&\le C_{m_0,z}\mathbb{E} \left[  \left(\frac{1}{N}  \int_0^{r_2} \left\|  (\mathrm{I} - \Delta)^{-\frac{\delta}{2}}  \nabla e^{(s-u)\Delta} (\mathrm{I} - \Delta)^{\frac{\beta+\delta}{2}}\left( e^{(t-s)\Delta}V_N - V_N \right) \right\|^2_{L^z(\mathbb R^d)} \ du \right)^{m_0/2}  \right] \\ 
&\le C_{m_0,z} \left(C_{\Delta}^\prime\right)^{m_0} \mathbb{E} \left[  \left(\frac{1}{N}  \int_0^{r_2}   (s-u)^{\delta-1}\left\|(\mathrm{I} - \Delta)^{\frac{\beta+\delta}{2}}\left( e^{(t-s)\Delta}V_N - V_N \right) \right\|^2_{L^z(\mathbb R^d)} \ du \right)^{m_0/2}  \right]. 
\end{align*}
Now, by the fundamental theorem of calculus and the heat semigroup estimate \eqref{eq:heat_semigroup_estimate}, we observe that
\begin{align*}
    \|e^{(t-s)\Delta} f - f \|_{L^z(\mathbb R^d)} \le \int_0^{t-s} \|\nabla \cdot e^{\tau\Delta} \nabla f\|_{L^z(\mathbb R^d)} \, d\tau \le  C_\Delta (t-s)^{1/2} \|f\|_{H^{1,z}(\mathbb R^d)}.
\end{align*}
Therefore, by choosing $f=(\mathrm{I} - \Delta)^{\frac{\beta+\delta}{2}} V_N$ and applying \eqref{Est: scaling_VN} with $\sigma = \beta + \delta + 1 \le 2$, we obtain
\begin{align*}
    \|I_2(0,r_2,\cdot)\|^{m_0}_{L^{m_0}(\Omega; H^{\beta,z}(\mathbb R^d))} & \le C_{m_0,z} \left(2 d C_\Delta C_{\Delta}^\prime\right)^{m_0} \mathbb{E}\left[\left(\frac{1}{N} \int_0^{r_2} \frac{t-s}{(s-u)^{1-\delta}}  \left\|V_N\right\|_{H^{\beta+\delta+1,z}(\mathbb R^d)}^2 \ du\right)^{m_0/2}\right].
\end{align*}
Let $C_{2,1}=2 d C_\Delta C_{\Delta}^\prime C_{m_0,z}^{1/m_0} C_{\beta+\delta+1,z} \delta^{-1/2} T^{\delta/2}\|V\|_{H^{2,z}(\mathbb R^d)}$. By letting $r_2 \to s$, we get
\begin{align}\label{Est: I_2 part 1}
     \|I_2(0,s,\cdot)\|_{L^{m_0}(\Omega; H^{\beta,z}(\mathbb R^d))} \le N^{-1/2+\alpha(d+\beta-d/z+\delta+1)} C_{2,1}(t-s)^{1/2}.
\end{align}

However, by the contractivity of the heat semigroup in $L^z(\mathbb R^d)$ and \eqref{Est: scaling_VN} with $\sigma = \beta + \delta < 2$, one also has that
\begin{align*}
    \left\|(\mathrm{I} - \Delta)^{\frac{\beta+\delta}{2}}\left( e^{(t-s)\Delta}V_N - V_N \right) \right\|_{L^z(\mathbb R^d)} &\le 2 \left\|(\mathrm{I} - \Delta)^{\frac{\beta+\delta}{2}} V_N \right\|_{L^z(\mathbb R^d)} \\
    &\le 2 C_{\beta+\delta,z} N^{\alpha(d+\beta-d/z+\delta)} \|V\|_{H^{2,z}(\mathbb R^d)}.
\end{align*}
Therefore, by defining $C_{2,2}=2 C_{\Delta}^\prime C_{m_0,z}^{1/m_0} C_{\beta+\delta,z} \delta^{-1/2} T^{\delta/2}\|V\|_{H^{2,z}(\mathbb R^d)}$, we have
\begin{align}\label{Est: I_2 part 2}
     \|I_2(0,s,\cdot)\|_{L^{m_0}(\Omega; H^{\beta,z}(\mathbb R^d))} \le N^{-1/2+\alpha(d+\beta-d/z+\delta)} C_{2,2}.
\end{align}
For any $\theta \in (0, 1-\beta]$, interpolating between \eqref{Est: I_2 part 1} and \eqref{Est: I_2 part 2} yields
$$ \|I_2(0,s,\cdot)\|_{L^{m_0}(\Omega; H^{\beta,z}(\mathbb R^d))} \le N^{-1/2+\alpha(d+\beta-d/z+\delta+\theta)} C_{2}(t-s)^{\theta/2}, $$
where $C_2 = C_{2,1}^\theta C_{2,2}^{1-\theta}$.

Finally, by choosing $\delta=\theta\in (0,1-\beta]$ we conclude that
$$\|S_N(t,\cdot)-S_N(s,\cdot)\|_{L^{m_0}(\Omega;H^{\beta,z}(\mathbb R^d))} \le N^{-1/2+\alpha(d+\beta-d/z+2\theta)}C^*(t-s)^{\theta/2},$$
where $C^* = C_1 + C_2$.

Assuming first that $\varepsilon \in (0, 2\alpha(1-\beta)]$ so that $\theta = \varepsilon/(2\alpha) \in (0, 1-\beta]$ (the case $\varepsilon > 2\alpha(1-\beta)$ follows immediately since $N \ge 1$), and setting $\eta = \theta/2$, we fix $m_0 \ge \max(1, m)$ large enough such that $m_0\eta > 1$. Applying Corollary 4.4 in \cite{2021_Richard}, we obtain the bound in $L^{m_0}(\Omega)$. By H\"older's inequality, 
$$\left\|S_N\right\|_{L^m(\Omega; C([0,T];H^{\beta,z}(\mathbb R^d)))} \le \widetilde{C}N^{-1/2+\alpha(d+\beta - d/z)+\varepsilon},$$
where $\widetilde{C}$ depends only on $C^*, m, m_0, \eta, T$.

\subsection{Proof of Proposition \ref{Prop:widetilde_S_NinL}}

For $z\in [2,+\infty)$, the assertion follows from Proposition \ref{Prop:est_stochastic_convolution_bessel} with $\beta=0$. For \(z=1\), we follow the weighted \(L^2\) argument in \cite{2023_oliveira_richard_tomasevic} underlying Lemma~\ref{Lemma:ORT}, using
\[ \|f\|_{L^1}\le \|w^{-1}\|_{L^2}\|wf\|_{L^2},
\qquad w(x)=(1+|x|)^{(d+1)/2}.
\]
The bounded drift in \eqref{SDE:MesoReg} and \eqref{Initial condition v.a.} provide uniform moments of all orders for the particle positions, controlling
the random translations. The spatial estimates require only weighted norms of
\(V\) and \(\nabla V\), bounded second derivatives, and fractional Sobolev norms of small positive order. These are controlled by compact support, \(V\in C_c^2(\mathbb R^d)\), and \eqref{Est: scaling_VN}.
The increment estimates and the Garsia--Rodemich--Rumsey argument therefore give
\[
\|\widetilde S_N\|_{L^m(\Omega;C([0,T];L^1(\mathbb R^d)))}
\le C N^{-1/2+\alpha d/2+\varepsilon}.
\]
Interpolation with the \(L^2\) estimate concludes the proof for \(1<z<2\).

\section*{Acknowledgments}
The authors are  members of GNAMPA (Gruppo Nazionale per l’Analisi Matematica, la Probabilità e le loro Applicazioni) of the Italian Istituto Nazionale di Alta Matematica (INdAM).

\printbibliography

\end{document}